\documentclass{article}

\usepackage[utf8]{inputenc}
\usepackage[english]{babel}
\usepackage{amssymb}
\usepackage{amsthm}
\usepackage{bbm}
\usepackage{commath}
\usepackage{enumitem}
\usepackage{geometry}
\usepackage{hyperref}
\usepackage{mathrsfs}
\usepackage{mathtools}
\newtheorem{thm}{Theorem}[section]
\newtheorem*{thm*}{Theorem}
\newtheorem{cor}[thm]{Corollary}
\newtheorem{prop}[thm]{Proposition}
\newtheorem{lem}[thm]{Lemma}

\theoremstyle{definition}
\newtheorem{defn}[thm]{Definition}

\newtheorem{exmp}[thm]{Example}

\newtheorem{notn}[thm]{Notation}

\theoremstyle{remark}
\newtheorem{rem}[thm]{Remark}

\makeatletter
\let\c@equation\c@thm
\makeatother
\numberwithin{equation}{section}

\newcommand{\bbA}{\mathbb{A}}
\newcommand{\bbB}{\mathbb{B}}
\newcommand{\bbC}{\mathbb{C}}

\newcommand{\bbN}{\mathbb{N}}
\newcommand{\bbP}{\mathbb{P}}
\newcommand{\bbQ}{\mathbb{Q}}
\newcommand{\bbR}{\mathbb{R}}

\newcommand{\bbX}{\mathbb{X}}
\newcommand{\bbZ}{\mathbb{Z}}

\newcommand{\cD}{\mathcal{D}}
\newcommand{\cH}{\mathcal{H}}
\newcommand{\cO}{\mathcal{O}}

\newcommand{\cS}{\mathcal{S}}

\newcommand{\sM}{\mathscr{M}}
\newcommand{\sN}{\mathscr{N}}

\newcommand{\sP}{\mathscr{P}}

\newcommand{\sS}{\mathscr{S}}

\newcommand{\fp}{\mathfrak{p}}

\newcommand{\FPart}[1]{\langle #1\rangle}

\DeclareMathOperator{\Hom}{Hom}

\DeclareMathOperator{\Ker}{Ker}

\DeclareMathOperator{\Res}{Res}

\DeclareMathOperator{\Supp}{Supp}

\newcommand{\ivec}[1]{
	\left(\begin{smallmatrix} #1 \end{smallmatrix}\right)}
\newcommand{\floor}[1]{\lfloor #1 \rfloor}

\DeclareMathOperator{\SL}{SL}

\DeclareMathOperator{\GL}{GL}

\DeclareMathOperator{\Ord}{ord}

\DeclarePairedDelimiterX{\inp}[2]{\langle}{\bbRangle}{#1, #2}

\DeclareMathOperator{\dlog}{\boldsymbol{dlog}}
\DeclareMathOperator{\dtau}{\boldsymbol{d\tau}}

\newcommand{\KSunit}[2]{\prescript{}{#1}{g}_{#2}}

\newcommand{\Vadele}[1]{V_{\bbA}^{#1}}
\newcommand{\DedeRa}{\boldsymbol{DR}}
\title{Explicit Cocycle of the Dedekind-Rademacher Cohomology Class and the Darmon-Dasgupta Measures}
\author{Jae Hyung Sim}

\begin{document}
\maketitle 

\begin{abstract}
	The work of Darmon, Pozzi, and Vonk \cite{DPVdr} has recently 
	shown that the RM-values of the Dedekind-Rademacher cocycle $J_{DR}$ 
	are Gross-Stark units up to controlled torsion. The authors
	of \cite{DPVdr} remarked that the measure-valued cohomology class,
	which we denote $\mu_{DR}^p$, underlying $J_{DR}$ is the level 1 incarnation
	of earlier constructions in \cite{DD}. In this paper,
	we make this relationship explicit by computing a concrete
	cocycle representative of an adelic incarnation $\mu_{DR}$ 
	by tracing the construction of the cohomology class and 
	comparing periods of weight 2 Eisenstein 
	series. While maintaining a global perspective in our computations, we
	configure the appropriate method of smoothing cocycles
	which exactly yields the $p$-adic measures of \cite{DD} when 
	applied to $\mu_{DR}$. These methods will also explain
	the optional degree zero condition imposed in \cite{DD} 
	which was remarked upon in \cite{DK} and \cite{FLcomp}.
\end{abstract}

\tableofcontents
\vspace{4mm}
In \cite{DVsingmoduli}, Darmon and Vonk introduced the theory of rigid cocycles
which drew analogies from classical Complex Multiplication theory to address
previously inaccessible questions regarding the arithmetic of real quadratic 
fields. In \cite{DPVdr}, Darmon, Pozzi, and Vonk used the deformation of Hilbert 
Eisenstein series to show that the Dedekind-Rademacher cocycle yields Gross-Stark 
units when evaluated at real quadratic points. The key component in the 
construction of the Dedekind-Rademacher cocycle is $\mu_{DR}^p$ which we
refer to as the Dedekind-Rademacher cohomology class. $\mu_{DR}^p$ is a cohomology 
class with $p$-adic measure coefficients whose $p$-adic 
Poisson Transform yields the analytic cocycle $J_{DR}$.

On the other hand, in \cite{DD}, Darmon and Dasgupta used the congruences of 
Dedekind sums arising as periods of families of Eisenstein series to construct
$\mu_{DD,\delta}$, a family of $p$-adic measures relating to the norm of a 
$p$-adic unit conjectured to be Gross-Stark units. Theorem~4.2 
of \cite{DD} shows the existence of a cocycle of $p$-adic measures whose values 
on homogeneous polynomials are given by periods of Eisenstein series. The 
Poisson transform of these $p$-adic measures were then proven to
yield Brumer-Stark units in Theorem~1.6 of \cite{DK} based on
Theorem~3.8 of \cite{DasShin}.

In \cite{Das}, Dasgupta leveraged the congruences of the Dedekind sums
and explicit polynomials converging to polynomials on open balls to give
an explicit formula for $\mu_{DD,\delta}$ on compact open sets. This formula
is given in terms of Dedekind sums involving the first periodified Bernoulli 
distribution. These formulae not only gave a simple algorithm for computing 
the conjectured Stark units but also proved the integrality of the measure by 
observation. In \cite{DasShin}, Dasgupta generalized these constructions using
Shintani zeta functions to define $p$-adic measures which generate 
Gross-Stark units as a corollary to the proof of the tame refinement of
the Gross-Stark conjecture for totally real fields \cite{DK}.

In \cite{DPVdr}, it is remarked that $\mu_{DD,\delta}$ is a sum of 
$\GL_2(\bbQ)$-translates of $\mu_{DR}^p$ so that $\mu_{DR}^p$ can 
be seen as a level 1 construction which generalizes 
$\mu_{DD,\delta}$. The main theorem of this paper makes this 
relationship precise by working with $\mu_{DR}$ which is an 
adelic generalization of $\mu_{DR}^p$ and ``smoothing'' by
an operator to obtain $\mu_{DR}^{p\delta}$ corresponding 
to the cohomology class of $\mu_{DD,\delta}$. Precisely, we have
the following theorem.
\begin{thm*}[Theorem~\ref{thmmaincomparison}]
	Let $c>1$ be a positive integer coprime to $6$, $N$ an integer coprime to $c$,
	and $\delta\in \bbZ[\GL_2(\bbQ)]$ be an element of the form $\delta=
	\sum_{D\mid N,D>0} n_D\cdot [LR(D)]$ satisfying $\sum_D n_D\cdot D=0$. 
	Then,$-(\mu_{DD,\delta})_c$ belongs to the cohomology class
	of $12\cdot \mu_{DR}^{p\delta}$.
\end{thm*}

The general strategy will be to use the $q$-expansion of the 
Kato-Siegel distribution to construct a cocycle $\DedeRa_\sS$
representing $\mu_{DR}$. Smoothing this cocycle
by an auxilliary $\delta$ will yield $\DedeRa^\delta$, a cocycle 
of a congruence subgroup which visibly coincides with $\mu_{DD,\delta}$ 
up to an explicit coboundary. From the necessary computation of periods 
of weight 2 Eisenstein series in the definition of $\DedeRa_\sS$ and
$\DedeRa^\delta$, we will also extract a simple formula for $\DedeRa_\sS$
with respect to test functions.

Another consequence of this paper is an explanation for the dispensable
degree zero condition on $\delta$ from \cite{DD}. Fleischer, Liu, and Dasgupta
observed that this condition was not necessary on the level of Eisenstein
series during their computation of elliptic units \cite{FLcomp}. The fact that 
Theorem~\ref{thmmaincomparison} also does not require the degree 0 further 
bolsters this observation. We discuss this further in Remark~\ref{remdivisor2}.

In \S1, we establish our conventions for discussing distributions on 
$V:=\bbQ^2$ and their adelic counterparts. We lay out methods of smoothing and 
restricting distributions and their resulting effect on group actions.

\S2 is dedicated to recalling the cohomological formulation of the Dedekind-Rademacher 
cohomology class and the Darmon-Dasgupta cocycles. Here, we will define $\mu_{DR}$,
$\mu_{DD,\delta}$, and the various decorations we adorn these two objects.
Our construction of the Dedekind-Rademacher cohomology class will largely follow
\cite{DPVdr}, but we will use the language established in \S1 to emphasize the 
adelic nature of the distributions. Since $\mu_{DD,\delta}$ is a family of
$p$-adic measures, we will be forced to descend $\mu_{DR}^\delta$ down to 
a $p$-adic measure $\mu_{DR}^{p\delta}$ to make a direct comparison in 
Theorem~\ref{thmmaincomparison}.

In \S3, we will detail the computation of explicit formulae for the
cocycles $\DedeRa_\sS$ and $\DedeRa^\delta$ which are respectively representatives
of $\mu_{DR}$ and $\mu_{DR}^\delta$. We first decompose $\mu_{KS}$ 
into a product of three simple distributions to define $\DedeRa_\sS$ 
and $\DedeRa^\delta$ in explicit fashion. These formulae are virtually identical
to those appearing in \cite{KLmu}, but we include the details for completeness. 
We then review the computations of periods of weight 2 Eisenstein series.
While we include some details, we refer to \cite{Stevens} for key theorems. 
We then compute $\DedeRa^\delta$ which proves Theorem~\ref{thmmaincomparison}
and conclude the paper by characterizing $\DedeRa_\sS$ by its values on 
generators of $G_\sS$ in Theorem~\ref{thmdrformula}.

\begin{notn}
	\begin{itemize}
		\item For an abelian group $G$, $G':=G-\{e\}$ where 
			$e$ is the identity element of $G$.
		\item For any rational number $x\in \bbQ$, $UL(x), LR(x), 
			D(x)\in\GL_2(\bbQ)$ are the diagonal matrices
			\begin{align*}
				UL(x)&= \begin{bmatrix} x&0\\
				0&1\end{bmatrix}
				& LR(x) &= \begin{bmatrix} 1&0\\
				0&x\end{bmatrix}
				&D(x)&= \begin{bmatrix} x&0\\
				0&x\end{bmatrix}.
			\end{align*}
		\item For any group $G$, $\gamma\in G$, $A$ a $G$-module, and $a\in A$,
			left actions will be denoted $\gamma\ast a$ while right actions
			will be denoted $a\mid \gamma$.
		\item $V=\bbQ^2$ denotes a 2-dimensional $\bbQ$-vector space with the standard basis.
		\item For a prime $\ell$, $V_\ell:=V\otimes_\bbQ \bbQ_\ell$.
		\item For a set of primes $\sS$, $\bbA^\sS:=\prod_{\ell\notin S}' \bbQ_\ell$.
		\item For any set $U$, $[U]$ is the characteristic function on $U$.
	\end{itemize}
\end{notn}

\section{Distribution}

In this section, we establish our conventions and methods
regarding distributions.

\subsection{Basic Definitions and Conventions}

\begin{defn}
	Let $\ell$ be a rational prime and $V_\ell=\bbQ_\ell^2$.
	A test function on $V_\ell$ is a function $f:V_\ell\to \bbZ$
	that is locally constant with compact support. We denote
	by $\cS(V_\ell)$ the abelian group of test functions.
	We denote by $f_\ell^0$ the characteristic function of 
	$\bbZ_\ell^2 \subset V_\ell$.

	For an abelian group $A$, $\cD(V_\ell,A):=\Hom_\bbZ(\cS(V_\ell),A)$
	is the group of $A$-valued distributions on $V_\ell$. 
\end{defn}

Considering $V_\ell$ as row vectors, we have a natural
right action of $\GL_2(\bbQ_\ell)$. For the sake of later convenience,
we convert this action to a left action so that for $\gamma\in 
\GL_2(\bbQ_\ell)$ and $v\in V_\ell$, $\gamma\ast v:= v\gamma^{-1}$. 
This induces a right action of $\GL_2(\bbQ_\ell)$ on $\cS(V_\ell)$
where $(f\mid \gamma)(v)=f(\gamma\ast v)$ for all $f\in \cS(V_\ell)$
and $\gamma\in \GL_2(\bbQ_\ell)$. The benefit of this slightly awkward
group action is evident when considering test functions that are
characteristic functions of subsets of $V_\ell$. For $U\subset V_\ell$ 
and $[U]$ its characteristic function, then for all $\gamma\in 
\GL_2(\bbQ_\ell)$ we have
\[
	([U]\mid \gamma)(v)=[U](v\gamma^{-1})=[U\gamma](v).
\]

Suppose now that $A$ is a right $\GL_2(\bbQ_\ell)$-module. Then,
$\cD(V_\ell,A)$ inherits a left $\GL_2(\bbQ_\ell)$ action where
for all $\mu\in\cD(V_\ell,A)$, $\gamma\in \GL_2(\bbQ_\ell)$,
and $f\in \cS(V_\ell)$,
\[
	(\gamma\ast\mu)(f)=\mu(f\mid \gamma)\mid \gamma^{-1}.
\]

\begin{defn}
	Let $\sS$ be an arbitrarily large set of rational primes. 
	The group of $\sS$-adelic test functions $\cS(\Vadele{\sS})$
	is defined as the restricted tensor product over \textit{finite} 
	places
	\[
		\sideset{}{'}\bigotimes_{\ell\notin \sS} \cS(V_\ell)
	\]
	where the restriction is that simple tensors are
	of the form $\otimes_{\ell\notin \sS} f_\ell$ where
	$f_\ell=f_\ell^0$ for all but finitely many primes.

	For an abelian group $A$, the group of $A$-valued
	$\sS$-adelic distributions is $\cD(\Vadele{\sS},A):= \Hom_\bbZ(\cS(\Vadele{\sS}),A)$.
\end{defn}

Letting $\bbA^\sS$ be the set of $\sS$-adeles, i.e. the set of
adeles ignoring the primes in $\sS$, $\cS(\Vadele{\sS})$ has a natural
right-action of $\GL_2(\bbA^\sS)$. If $A$ is a right 
$\GL_2(\bbA^\sS)$-module, then $\cD(\Vadele{\sS},A)$ inherits a left 
$\GL_2(\bbA^\sS)$-action analogous to the $\ell$-adic case. 

From now on, let $G:=\GL_2^+(\bbQ)$. We note that for any set of 
primes $\sS$, we can consider $G$ as a subgroup of $\GL_2(\bbA^\sS)$
by the diagonal embedding.

\begin{exmp}
	\begin{itemize}
		\item Let $\bbZ$ be the right $G$-module with the trivial $G$-action.
			Then, for $\gamma\in G$, $f\in \cS(\Vadele{\sS})$, and $\mu\in
			\cD(\Vadele{\sS},\bbZ)$, we have $(\gamma\ast \mu)(f)=\mu(f\mid \gamma)$.
		\item Let $\cO_\cH$ be the additive group of holomorphic functions
			on $\cH$ the complex upper-half plane. $\cO_\cH$ has a right
			$G$-module structure given by Moebius transform on $\cH$. For 
			any set of primes $\sS$, $\cD(\Vadele{\sS},\cO_\cH)$ inherits a left 
			$G$-module structure. Explicitly, if $\gamma\in G$, $f\in
			\cS(\Vadele{\sS})$, $\mu\in \cD(\Vadele{\sS},\cO_\cH)$, and $\tau\in \cH$, 
			we have
			\[
				(\gamma\ast\mu)(f)(\tau)=\mu(f\mid \gamma)(\gamma^{-1}\ast \tau).
			\]
		\item The same observation can be made for $\cO_\cH^\times$, the 
			multiplicative group of units of $\cO_\cH$.
	\end{itemize}
\end{exmp}

\begin{defn}
	Let $V:=\bbQ^2$. A lattice in $V$ is a rank 2 $\bbZ$-submodule
	of $V$. We endow $V$ with the lattice topology, i.e. the topology
	generated by lattices in $V$. For a function $f:V\to \bbZ$,
	we define the following terms:
	\begin{itemize}[topsep=0pt]
		\item The support of $f$ is the set $\Supp(f):=
			\{v\in V\mid f(v)\neq 0\}$,
		\item $f$ is bounded if there exists a lattice
			$\Lambda$ such that $\Supp(f)\subset \Lambda$,
		\item $f$ is $\Lambda$-invariant for a lattice $\Lambda$
			if for all $v\in V$ and $\lambda\in\Lambda$,
			$f(v)=f(v+\lambda)$.
	\end{itemize}

	The set of test functions on $V$ is the set $\cS(V):=
	\{f:V\to \bbZ\mid f\text{ is bounded and locally constant.}\}$.
	We write $\cS(V')\subset \cS(V)$ for the subgroup of test
	functions supported away from $0$. For a lattice $\Lambda$,
	$\cS_\Lambda(V)$ is subgroup of test functions that are 
	$\Lambda$-invariant.

	For an abelian group $A$, $\cD(V,A):=\Hom(\cS(V),A)$ is the
	group of $A$-valued distributions on $V$.
\end{defn}

\begin{notn}
	In general, for any subspace/subgroup $U$ of a topological
	abelian group over which we define test functions 
	and distributions (e.g. $V_\bbA^\sS$ and $V_\ell$), for 
	any abelain group $A$, $\cS(U')\subset \cS(U)$ will denote
	the subset of test functions supported away from the 
	identity element and $\cD'(U,A)$ will be shorthand for 
	$\Hom_\bbZ(\cS(U'),A)$.
\end{notn}

As in the local setting, $V$ has a right action of $\GL_2(\bbQ)$
where $V$ is thought of as row vectors which we similarly convert
to a left action. This endows $\cS(V)$ (resp. $\cD(V,A)$) a right
(resp. left) action of $\GL_2(\bbQ)$. 

The adelic distributions and the distributions on $V$ are connected
by the following proposition, which will be especially helpful
in giving explicit formulae for adelic distributions.

\begin{prop}
	Let $\phi:\cS(\Vadele{\emptyset})\to \cS(V)$ be the homomorphism 
	defined by $\phi(\otimes_\ell f_\ell) = \prod_\ell f_\ell\mid_V$ where
	the restriction to $V$ is the restriction to the natural 
	inclusion of $V$ in $V_\ell$. Then, $\phi$ is an isomorphism
	of $\GL_2(\bbQ)$-modules.
\end{prop}

\begin{proof}
	We define an explicit inverse $\psi:\cS(V)\to \cS(\Vadele{\emptyset})$. 
	For any affine lattice $x+\Lambda\subset V$, we assign
	\[
		\psi([x+\Lambda])=\otimes_\ell [x+\Lambda\otimes \bbZ_\ell].
	\]
	Notice that for all but a finite number of primes, 
	$x\in \bbZ_\ell^2$ and $\Lambda\otimes \bbZ_\ell = \bbZ_\ell^2$. 

	We claim that $\phi\circ \psi$ is the identity function
	on $\cS(V)$. It is sufficient to see this on characteristic 
	functions of affine lattices. Let $[x+\Lambda]\in \cS(V)$. 
	If $\alpha\in x+\Lambda$, $[x+\Lambda\otimes \bbZ_\ell](\alpha)=1$
	for all primes $\ell$.
	If $\alpha\notin x+\Lambda$, we want to show that there exists
	a prime $\ell$ such that $[x+\Lambda\otimes\bbZ_\ell](\alpha)=0$.
	It is sufficient to show that $\alpha-x\notin \Lambda\otimes\bbZ_\ell$
	for some prime $\ell$. In fact, by scaling each coordinate, it is 
	sufficient to show this for the case when $\Lambda=\bbZ^2$, which is 
	true since an element of $V$ is in $\bbZ^2$ if and only if it is 
	integral at all primes.

	To show that $\phi$ is an isomorphism, we apply the strong
	approximation theorem to see that for all tensors of the 
	form \[
		\otimes_\ell [x_\ell+\Lambda_\ell]
	\]
	where $[x_\ell+\Lambda_\ell]=f_\ell^0$ for all but a finite number
	of primes, there exists $x\in V$ and $v_1,v_2\in V$ such that
	$[x+v_1\bbZ_\ell+v_2\bbZ_\ell]=[x_\ell+\Lambda_\ell]$ 
	for all primes $\ell$. This means $\psi$ is surjective, so
	$\phi$ must be an isomorphism.
\end{proof}

\begin{rem}
	All definitions above can be replicated for the one-dimensional
	setting by replacing $V$ with $\bbQ$ and $\GL_2$ with $\GL_1$.
	The one-dimensional setting will be used to define various
	distributions in \S~\ref{distex}. The details 
	are entirely analogous and are left as an exercise for the reader.
\end{rem}

\subsubsection{Smoothing}

Let $\delta=\sum_g n_g\cdot [g]\in \bbZ[G]$.

\begin{defn}\label{defndeltasmoothing}
	Let $M$ be a left $G$-module. The $\delta$-smoothing map is defined by
	\begin{align*}
		-^\delta:M&\to M\\
		m&\mapsto m^\delta:=\delta\ast m.
	\end{align*}
\end{defn}

\begin{prop}\label{propdeltasmoothing}
	Let $H\subset G$ be a subgroup. Define $H_\delta\subset H$ as the intersection 
	\[
		H_\delta=\bigcap_{\substack{g\in G\\ n_g\neq 0}} gHg^{-1}\cap H.
	\]
	The $\delta$-smoothing map sends $H$-invariant elements to $H_\delta$-invariant
	elements,
	\[
		-^\delta:M^H\to M^{H_\delta}.
	\]
\end{prop}

\begin{proof}
	Let $m\in M^H$. Let $\gamma\in H_\delta$ and for all $g\in G$
	such that $n_g\neq 0$, let $\gamma_g\in H$ such that 
	$\gamma = g\gamma_g g^{-1}$. We then observe that
	\[
		\gamma\ast m^{\delta}=\sum_{g} n_g\cdot(\gamma g)\ast m
		=\sum_g n_g\cdot (g\gamma_g)\ast m=\sum_g n_g\cdot g\ast m
		=m^\delta.
	\]
\end{proof}

\begin{exmp}\label{exmpcsmoothing}
	Let $A$ be a right $G$-module and let $c\in \bbZ$ be a nonzero integer.
	For $\mu\in \cD(\Vadele{\sS},A)$, the $c$-smoothed distribution 
	$\mu_c\in \cD(\Vadele{\sS},A)$ is the distribution obtained by smoothing 
	$\mu$ by the group element $c^2[Id]-[D(c)]$. Note that since $Id$ 
	and $D(c)$ are both in the center of $G$, $c$-smoothing is a 
	$H$-morphism for any subgroup $H\subset G$.
\end{exmp}

\begin{exmp}
	Let $\sS$ be a set of primes and $G_\sS:=\GL_2^+(\bbZ_\sS)$
	as defined in Definition~\ref{defnGS}. Let $N$ be an integer 
	supported over $\sS$ and $\delta=\sum_{D\mid N} n_D\cdot [LR(D)]$. 
	Then, letting $G_\sS(N)\subset G_\sS$ be the congruence subgroup 
	of matrices that are upper-triangular modulo $N$, we get a 
	$\delta$-smoothing map $\cD(\Vadele{\sS},A)^{G_\sS}\to 
	\cD(\Vadele{\sS},A)^{G_\sS(N)}$.

	We note that if $A$ is acted on trivially by scalar matrices, 
	the equality $LR(D)^{-1}=UL(D)\cdot D(D^{-1})$ yields
	\begin{align}
		\mu^\delta(f)=\sum n_D\cdot (\mu(f\mid LR(D))\mid UL(D)). \label{eqndeltasmoothing}
	\end{align}
\end{exmp}

\subsubsection{Moving between sets of primes}

\begin{defn}
	Let $\sN\supset \sS$ be two sets of primes of arbitrary size. 
	We define an injective group homomorphism $i_\sN^\sS:
	\cS(\Vadele{\sN})\to \cS(\Vadele{\sS})$ where
	\[
		i_\sN^\sS(f)= f\otimes f_{\sN-\sS}^0.
	\]

	We define $\cS(V_\sS)\subset \cS(V)$ (resp. $\cS(V_\sS')\subset\cS(V')$) 
	as the image of $\cS(\Vadele{\sS})$ under $\phi\circ i_\sS^\emptyset$. For a 
	lattice $\Lambda\subset V$, $\cS_\Lambda(V_\sS)$ is the intersection $\cS_\Lambda(V)\cap 
	\cS(V_\sS)$. For notational convenience, we write $\cD(V_\sS',A)$ for 
	$\Hom(\cS(V_\sS'),A)$ and $\cD'(\Vadele{\sS},A)$ for 
	$\Hom((\phi\circ i_\sS^\emptyset)^{-1}(\cS(V_\sS')),A)$.
\end{defn}

\begin{defn}\label{defnGS}
	Define $\bbZ_\sS\subset \bbQ$ as the subring of rational numbers that are 
	integral at all primes in $\sS$. If $\sS$ is finite, and $n$ is the product
	of primes in $\sS$, $\bbZ_\sS$ is the localization $\bbZ_{(n)}$.
	Define $G_\sS$ as the group $\GL_2^+(\bbZ_\sS)$ and for any integer $N$ 
	supported over $\sS$, $G_\sS(N)\subset G_\sS$ is the congruence subgroup 
	of matrices that are upper-triangular modulo $N$.
\end{defn}

\begin{rem}\label{remVsS}
	The notation $\cS(V_\sS)$ is chosen to reflect the fact that 
	$\cS_{\bbZ^2}(V_\sS)$ is generated by the characteristic functions
	of the cosets $\bbZ_\sS^2/\bbZ^2$. We warn the reader that this
	cannot be extended to other lattices. For example, if $\ell\in \sS$,
	the characteristic function $f=[(\ell,\ell)+\ell^2\bbZ^2]$ is not 
	contained in $\cS(V_\sS)$ since the $\ell$-factor of the simple
	tensor $\phi^{-1}(f)\in \cS(\Vadele{\emptyset})$ is not $f_\ell^0$.
\end{rem}

\begin{rem}
	The multitude of decorations on $V$ are meant to emphasize the
	different facets of the group $\cS(V)$. In particular, the form
	$\cS(V_\sS)$ is especially useful when calculating the values 
	of scalar-invariant distributions since $\cS_{\bbZ^2}(V_\sS)$
	has a nice set of generators. However, the $G_\sS$ action
	arises from $\cS(\Vadele{\sS})$, so we will use
	$\cS(\Vadele{\sS})$ when the $G_\sS$-action is relevant.
\end{rem}

\begin{prop}\label{propGSequivariance}
	Let $\sS\subset \sN$ be two sets of primes. $i_\sN^\sS$ is a
	$G_\sN$-module homomorphism.
\end{prop}

\begin{proof}
	We observe that for all $\ell\in \sN- \sS$ and $g\in G_\sN$,
	the action of $g$ on $f_\ell^0$ factors through $\GL_2(\bbZ_\ell)$,
	so $f_\ell^0\mid g = f_\ell^0$. Thus, for all $f\in \cS(\Vadele{\sN})$, 
	\[
		i_\sN^\sS(f\mid g)=(f\mid g)\otimes f_{\sN-\sS}^0
		=(f\otimes f_{\sN-\sS}^0)\mid g=i_\sN^\sS(f)\mid g.
	\]
\end{proof}

\begin{cor}
	Let $A$ be a right $G$-module. Then, for two sets of primes
	$\sS\subset\sN$, the pull-back map $(i_\sN^\sS)^\ast:
	\cD(\Vadele{\sS},A)\to \cD(\Vadele{\sN},A)$ is a $G_\sN$-module homomorphism.
\end{cor}

\subsubsection{Freeness of $\cS(V_\sS)$}

\begin{lem}\label{lemStestfree}
	For all lattices $L\subset V$, $\cS_L(V)$ and $\cS_L(V')$ 
	are free $\bbZ$-modules. The same is true for $\cS_L(V_\sS)$ and
	$\cS_L(V_\sS')$ for any set of primes $\sS$.
\end{lem}

\begin{proof}
	Since $\cS_L(V'),\cS_L(V_\sS)$, and $\cS_L(V_\sS')$ are subgroups
	of $\cS_L(V)$, it suffices to show that $\cS_L(V)$ is free.
	We write down an explicit basis for $\cS:=\cS_L(V)$.
	Define $B_L:=\{[v]\mid v\in V/L\}$. Then, for all test functions 
	$f\in \cS$, we can uniquely decompose $f$ as 
	$\sum_{v\in V/L} f(v)\cdot [v]$. (Notice that since $f$ is $L$-invariant,
	it makes sense to evaluate $f$ on $V/L$.) The uniqueness of this 
	decomposition is immediate. The sum is also finite since 
	$f$ has bounded support, meaning there exists a lattice $\Lambda
	\subset V$ such that $\Supp(f)\subset \Lambda$, so $f$ is
	supported on $\Lambda/L$ which is a finite set.
\end{proof}

\begin{prop}\label{proptestfree}
	For any set of primes $\sS$, $\cS(V_\sS)$ and $\cS(V_\sS')$ are free 
	abelian groups.
\end{prop}

\begin{proof}
	Once again, since $\cS(V_\sS)$ and $\cS(V_\sS')$ are subgroups
	of $\cS(V)$, it is sufficient to show that $\cS:=\cS(V)$ is 
	a free abelian group. For any lattice $\Lambda\subset V$, let
	$\cS_\Lambda:=\cS_\Lambda(V)$. For all $n\in \bbN$, define 
	$L_n:=n\bbZ^2$. All lattices contain $L_n$ for some $n$,
	so $\cS_\Lambda\subset \cS_{L_n}$ for some $n$. Thus, it is 
	sufficient to find a basis for the union 
	$\bigcup_{n=1}^\infty \cS_{L_n}$. 

	Define a directed system on $\{\cS_{L_n}\mid n\in \bbN\}$ partially
	ordered by divisibility on $\bbN$ where for all
	$n,m\in \bbN$, $j_{n,m}:\cS_{L_n}\to \cS_{L_{nm}}$ is the
	natural inclusion. Then, we have $\cS=\varinjlim_n \cS_{L_n}$. 
	To show that $\cS$ is free, it suffices to show that each 
	transition map $j_{n,m}$ is an isomorphism onto a direct summand. 

	Fix $n,m\in \bbN$. By the proof of Lemma~\ref{lemStestfree},
	we have an explicit basis of $\cS_{L_n}$ of the form $B_n
	=\{[v+L_n]\mid v\in V/L_n\}$. We want to show that this basis 
	can be extended to a basis of $\cS_{L_{nm}}$. Fix $v\in V$
	and we notice that
	\[
		[v+L_n] = \sum_{w\in L_n/L_{nm}}
		[v+w+L_{nm}].
	\]
	We can thus define a basis for the test functions in $\cS_{L_{nm}}$
	whose support lies in $v+L_n$ as follows:
	\[
		B_v:=\{[v+L_n]\}\cup\{[v+w+L_{nm}]\mid w\in (L_n/L_{nm})'\}.
	\]
	Then, $\cS_{L_{nm}}=\bigoplus_{v\in V/L_{n}} \bigoplus_{f\in B_v}
	f\cdot\bbZ$. By the construction of $B_v$, we can see that $j_{n,m}$ 
	identifies $\cS_{L_n}$ with a direct summand $\cS_{L_{nm}}$. 
\end{proof}

\subsubsection{Decomposition of $G_\sS$}

\begin{lem}\label{lem:gl2gens}
	Let $T_\sS\subset G_\sS$ be the subgroup of diagonal matrices. 
	Then, $G_\sS$ is generated by $T_\sS$ and $\SL_2(\bbZ)$.
\end{lem} 

\begin{proof}
	Notice that $G_\sS$ acts transitively on $\bbP^1(\bbQ)$.
	Since $\SL_2(\bbZ)$ also acts transitively on $\bbP^1(\bbQ)$,
	we only need to show that stabilizers of the infinity cusp,
	i.e. upper-triangular matrices, can all be written as a 
	product of elements in $T_\sS$ and $\SL_2(\bbZ)$.

	Take a generic upper-triangular matrix in $G_\sS$,
	\[
		\gamma=\begin{bmatrix}
			a&b\\ 
			0&d
		\end{bmatrix}.
	\]
	Since this matrix is invertible, we know that $a$ and $d$
	are units in $\bbZ_\sS$. Let $\alpha\in\bbZ_\sS^\times$ be 
	a positive unit such that $\alpha b/a$ is an integer. Then, 
	we have
	\[
		\gamma = \begin{bmatrix}
			a/\alpha&0\\
			0&d
		\end{bmatrix}
		\begin{bmatrix}
			1& \alpha b/a\\
			0&1
		\end{bmatrix}
		\begin{bmatrix}
			\alpha &0\\
			0&1
		\end{bmatrix},
	\]
	which is clearly a product of matrices in $T_\sS$ and $\SL_2(\bbZ)$.
\end{proof}

We note that if $N$ is an integer that is supported over $\sS$, 
the proof above can be exactly replicated to show that $G_\sS(N)$
(the congruence subgroup of matrices upper-triangular modulo $N$)
is generated by $T_\sS$ and $\Gamma_0(N)$. The only change required
is to notice that the orbit of $i\infty$ under the action of $G_\sS(N)$
is the union of $i\infty$ with rational numbers whose denominator is
divisible by $N$.

For computational purposes, we use the above proof to write an explicit decomposition
of a generic element of $G_\sS$ as a product of elements in $T_\sS$ and
$\SL_2(\bbZ)$. Note that we can use continued fractions to easily write down
elements of $\SL_2(\bbZ)$ as explicit products of generators.
\begin{cor}\label{corgendecomp}
	Let $\gamma\in G_\sS$ be of the form
	\[
		\gamma=\begin{bmatrix}a&b\\ c&d\end{bmatrix}.
	\]
	Let $\beta\in \bbZ_\sS^\times$ be a positive integer such that
	$a\beta$ and $c\beta$ are coprime integers and let $x,y\in \bbZ$
	such that $(a\beta)y-(c\beta)x=1$. Let $\alpha\in \bbZ_\sS^\times$
	be positive integer such that $\alpha(by-xd)\in \bbZ$ and
	let $D=\det(\gamma)$. Then, we have the decomposition
	\[
		\gamma = \begin{bmatrix}
			a\beta & x\\
			c\beta & y
		\end{bmatrix}
		\begin{bmatrix}
			(\alpha\beta)^{-1}&0\\
			0&D\beta
		\end{bmatrix}
		\begin{bmatrix}
			1& \alpha (by-xd)\\
			0&1
		\end{bmatrix}
		\begin{bmatrix}
			\alpha & 0\\
			0&1
		\end{bmatrix}.
	\]
\end{cor}

\subsection{Constructing Distributions}
The following proposition provides a tool for easily describing most
of the distributions we will encounter while simultaneously proving
their existence. 
\begin{prop}\label{propdistconstruction}
	Let $A$ be an abelian group with a right $\bbQ^\times$-action and
	let $L=\bbZ^2$. Given a function $\sigma:V/L\to A$, the following 
	two statements are equivalent.
	\begin{enumerate}
		\item There exists $\mu\in \cD(V,A)$ that is $\bbQ^\times$-invariant
			such that for all $v\in V/L$, $\mu([v])=\sigma(v)$.
		\item For all nonzero $t\in \bbZ$ and $v\in V/L$, we have
			the equality
			\[
				\sum_{\substack{w\in V/L\\ wt=v}}
				\sigma(w)=\sigma(v)\mid t^{-1}.
			\]
	\end{enumerate}
\end{prop}

\begin{proof}
	Suppose the first statement is true and let $v\in V/L$. For nonzero $t\in \bbZ$,
	we have $[v]\mid D(t^{-1})=[vD(t^{-1})]$. Since $vD(t^{-1})$ is $L$-uniform,
	we can decompose it as the disjoint union \[
		vD(t^{-1})=\bigcup_{\substack{w\in V/L\\ wt=v}} w.
	\]
	Thus, $\mu$ being $\bbQ^\times$-invariant gives us the following equalities.
	\begin{align*}
		\sigma(v)\mid t^{-1}&=\mu([v])\mid t^{-1}\\
		&=\mu([v]\mid D(t^{-1}))\\
		&=\mu\left( \sum_{w\in V/L;\, wt=v} [w]\right)\\
		&=\sum_{\substack{w\in V/L\\ wt=v}}\mu([w])\\
		&=\sum_{\substack{w\in V/L\\ wt=v}}\sigma(w).
	\end{align*}

	For the converse, we first define $\mu_L:\cS_L(V)\to A$ by 
	writing test functions $f\in \cS_L(V)$ uniquely as a sum of
	basis elements, i.e. $f=\sum_{v\in V/L} f(v)[v]$. We then assign $\mu_L(f)
	=\sum_{v\in V/L} f(v) \sigma(v)$. We extend this to a distribution
	by noticing that for all $f\in \cS(V)$, there exists $t\in \bbZ$
	such that $f\mid t^{-1}\in \cS_L(V)$, so we define $\mu(f)
	=\mu_L(f\mid t^{-1})\mid t$. We need to show that $\mu$ is 
	well-defined.

	The unique decomposition of elements in $\cS_L(V)$ 
	already makes $\mu_L$ well-defined, so
	it is sufficient to show that for all nonzero $t\in \bbZ$, 
	$\mu_L(f)=\mu_L(f\mid D(t^{-1}))\mid t$. In fact, we can reduce
	to the case when $f=[v]$ for some $v\in V/L$, which is
	exactly the second statement of the proposition.
\end{proof}

Now, suppose $A$ is a right $\bbZ_\sS^\times$-module. Then, we have
a corresponding statement for $\bbZ_\sS^\times$-invariant distributions
in $\cD(V_\bbA^\sS,A)$ and maps of the form $\sigma:\bbZ_\sS^2/\bbZ^2\to A$.
We note that, essentially due to Remark~\ref{remVsS}, we only include
the case when the lattice is $\bbZ^2$.

\begin{cor}\label{cordistconstruction}
	Suppose $A$ is an abelian group with the right action of
	$\bbZ_\sS^\times$. Given a function $\sigma:\bbZ_\sS^2/\bbZ^2
	\to A$, the following two statements are equivalent:
	\begin{enumerate}
		\item There exists $\mu\in \cD(V_\sS,A)$ that is 
			$\bbZ_\sS^\times$-equivariant such that for all 
			$v\in \bbZ_\sS^2/\bbZ^2$, $\mu([v])=\sigma(v)$.
		\item For all nonzero $t\in \bbZ\cap \bbZ_\sS^\times$ and 
			$v\in \bbZ_\sS^2/\bbZ^2$, we have the equality
			\[
				\sum_{\substack{w\in \bbZ_\sS^2/\bbZ^2\\ wt=v}}
				\sigma(w)=\sigma(v)\mid t^{-1}.
			\]
	\end{enumerate}
\end{cor}

\begin{proof}
	The proof is identical to that of Proposition~\ref{propdistconstruction}
	by replacing $L$ with $\bbZ^2$ and observing that \[
		\cS_{\bbZ^2}(V_\sS)=\{[v]\mid v\in \bbZ_\sS^2/\bbZ^2\}.
	\]
\end{proof}

\subsubsection{Examples of Distributions}\label{distex}
The following distributions will be integral to our constructions and comparisons
in later sections. 

\paragraph{Dirac Distribution}
For $v\in V$, $\delta_v\in \cD(V,\bbZ)$ is the evaluation map
at $v$, i.e. $\delta_v(f)=f(v)$. We give special attention to
the rank 1 Dirac distribution $\delta_0\in \cD(\bbQ,\bbZ)$ which 
we will use extensively. 

\paragraph{Bernoulli Distributions}
Recall the Bernoulli polynomials defined by the following generating function.
\[
	\frac{te^{xt}}{e^t-1} = \sum_{n=0}^\infty B_n(x)\cdot \frac{t^n}{n!}.
\]
For $r$ a non-negative integer, the $r$-th periodified Bernoulli polynomial
$\bbB_r:\bbQ\to\bbQ$ is the function defined by
\[
	\bbB_r(x)=\begin{cases}
		0& \text{if }r=1 \text{ and } x\in \bbZ\\
		B_r(\langle x\rangle)& \text{otherwise}
	\end{cases}
\]
where $\langle x\rangle$ is defined as the fractional part of $x$.

\begin{prop}\label{propbernoullidist}
	Let $r$ be a non-negative integer and let $\bbQ_{r-1}$ be the 1-dimensional
	$\bbQ$-vector space equipped with the degree $r-1$ right action by $\bbQ^\times$,
	i.e. for all $x\in \bbQ$ and $n\in \bbQ^\times$, $x\mid m = m^{r-1}\cdot x$.
	Then, with some abuse of notation, there exists a 
	distribution $\bbB_r\in \cD(\bbQ,\bbQ_{r-1})$ such that
	for all $x\in \bbQ$, $\bbB_r([x+\bbZ])=\bbB_r(x)$.
\end{prop}

\begin{proof}
	Using the $1$-dimensional analogue of Proposition~\ref{propdistconstruction}
	means we only need to show that for all $m\in \bbN$,
	\[
		m^{1-r}\cdot \bbB_r(x) = \sum_{ym = x}\bbB_r(y).
	\]
	We do this for all $r$ simultaneously. Letting $0<x<1$ be a 
	rational number, we have
	\begin{align*}
		\sum_{r=0}^{\infty}\frac{t^r}{r!}\left( \sum_{i=0}^{m-1} \bbB_r((x+i)/m) \right) 
		&=\sum_{i=0}^{m-1}\frac{te^{(x+i)t/m }}{e^t-1}\\
		&=\frac{te^{xt/m}}{e^t-1}\cdot \sum_{i=0}^{m-1} e^{t i/m }\\
		&=\frac{te^{xt/m}}{e^t-1}\cdot \frac{e^t -1}{e^{t/m}-1}\\
		&=\frac{te^{xt/m}}{e^{t/m}-1}\\
		&=m\frac{(t/m)e^{x(t/m)}}{e^{t/m}-1}\\
		&=m\sum_{r=0}^{\infty}\bbB_r(x)\frac{t^r}{r!\cdot m^r} \\
		&=\sum_{r=0}^\infty m^{1-r}\bbB_r(x)\frac{t^r}{r!}.
	\end{align*}

	For the case when $x=\bbZ$, the calculation above should always carry a
	$t/2$ error term. Otherwise, the computation is identical and yields the
	required distribution relation.
\end{proof}

\begin{rem}
	As noted in Proposition~\ref{propbernoullidist}, we will
	often abuse notation to denote both the periodified
	Bernoulli polynomial and the associated distribution by
	$\bbB_r$ for all $r\ge 0$. 

	In addition, the non-periodified first Bernoulli 
	polynomial will naturally arise in our computations. 
	To account for these polynomials in a distribution 
	theoretic manner, we further abuse notation to define 
	$B_1\in \cD(\bbQ,\bbQ)$ as the distribution $B_1=
	\bbB_1 - \tfrac{\delta_0}{2}$
	where $\delta_0$ is the rank 1 Dirac distribution. 
	Since $\bbB_1$ and $\delta_0$ are both 
	$\bbQ^\times$-invariant, $B_1$ is also 
	$\bbQ^\times$-invariant.
\end{rem}

\paragraph{Cyclotomic Distributions}

To save space, we state the following lemma without proof, though 
we note the lemma holds when replacing $\bbC$ with $\bbQ^{ab}$ 
and $e^{2\pi i j/n}$ with a compatible system of roots of unity.

\begin{lem}\label{lemcycpoly}
	Let $R$ be a commutative $\bbC$-algebra. Then for all $n\in \bbN$
	and $x\in R$, we have
	\[
		\prod_{j=0}^{n-1} (1-xe^{2\pi i j/n})=1-x^n.
	\]
\end{lem}

\newcommand{\DCycdist}{\mathrm{DCyc}}
\newcommand{\Cycdist}{\mathrm{Cyc}}
\newcommand{\Cyclog}{\mathrm{Clog}}
\newcommand{\RCyclog}{\mathrm{RClog}}
\begin{cor}\label{corcycdist}
	Endow $\bbC$, $\bbR$, and $\bbC^\times$ with the trivial $\bbQ^\times$-action.
	Furthermore, for all $x\in \bbC$ in the closed unit disc but not equal to $1$,
	define $\log(1-x)$ via the power series\[
		\log(1-x)=-\sum_{n=1}^\infty \frac{x^n}{n}.
	\]
	Then, there exist unique $\bbQ^\times$-invariant 
	distributions $\Cycdist\in \cD'(\bbQ,\bbC^\times)$,
	$\Cyclog\in \cD'(\bbQ,\bbC)$, and $\RCyclog\in 
	\cD'(\bbQ,\bbR)$ such that for all $f=[\tfrac{a}{b}+\bbZ]
	\in \cS(\bbQ')$,
	\begin{itemize}
		\item $\Cycdist(f)=1-e^{2\pi i a/b}$,
		\item $\Cyclog(f)=\log(\Cycdist(f))$,
		\item $\RCyclog(f)=\log\abs{\Cycdist(f)}$.
	\end{itemize}
\end{cor}

\begin{proof}
	The existence of $\Cycdist$ is given by Proposition~\ref{propdistconstruction}
	and Lemma~\ref{lemcycpoly}. The existence of $\Cyclog$ and $\RCyclog$ are
	consequences of $\log$ and $\log\circ\abs{-}$ defining $\bbQ^\times$-module 
	homomorphisms from the image of $\Cycdist$ to $\bbC$ and $\bbR$, respectively.
\end{proof}

\begin{rem}
	We note that $\Cyclog - \RCyclog = \pi i B_1$.
\end{rem}

\paragraph{Tensor of Distributions}
\begin{defn}
	Let $\mu_1\in \cD(\bbQ,M_1)$ and $\mu_2\in \cD(\bbQ,M_2)$
	for some $\bbQ^\times$-modules $M_1,M_2$. Then, $\mu_1\otimes \mu_2
	\in \cD(V,M_1\otimes M_2)$ is the distribution where for 
	test functions of the form $f=[U_1\times U_2]$, $(\mu_1\otimes \mu_2)(f)
	=\mu_1([U_1])\otimes \mu_2([U_2])$.
\end{defn}

\begin{exmp}\label{exmpdsumdcyc}
	Notice that for any $n,m\in \bbZ_{\ge 0}$, $\bbQ_n\otimes 
	\bbQ_m \cong \bbQ_{n+m}$ by the map $1\otimes 1\mapsto 1$. 
	This shows that we can define distributions of the form 
	$\bbB_n\otimes\bbB_m \in \cD(V,\bbQ_{n+m-2})$. 

	Using the cyclotomic distributions, we define $\DCycdist
	\in \cD'(V,\bbC^\times)$ as the composition 
	$e^{\delta_0\otimes\Cyclog}$.
\end{exmp}

\paragraph{Theta-unit Distribution}
The next definition establishes a ring of $q$-expansions which we will
use to define the Theta-unit distribution and decompose the Kato-Siegel
distribution in \S3.

\begin{defn}\label{defnqexp}
	For all positive integers $n$, let $\cO_{\cH,n}$
	denote the ring of holomorphic functions on $\cH$
	invariant under translation by $n$. Then, we have
	an embedding $\cO_{\cH,n}\to \bbC[[q_n]]$ via Fourier
	expansion where $q_n=e^{2\pi i \tau/n}$. 
	Fixing an embedding $\bbQ^{ab}\to \bbC$ and denoting
	$\zeta^x$ for $e^{2\pi i x}$ for all $x\in \bbQ$, let 
	$P_n \subset \bbC[[q_n]]$ be the image of $\cO_{\cH,n}$
	lying in $\bbQ^{ab}[[q_n]]$ and let $F_n$ be the 
	field of fractions of $P_n$. Define $Q_n$ as the 
	multiplicative subgroup of $F_n^\times$ generated
	by $q_n$ and constants and define $U_n$ as the
	multiplicative subgroup of expansions whose constant
	term is $1$.

	For all $n,m\in \bbN$, the natural inclusion $\cO_{\cH,n}
	\to \cO_{\cH,nm}$ extends to a function $i_{n,m}:F_n\to 
	F_{nm}$ characterized by $i_{n,m}(q_n)=q_{mn}^m$. 
	We define $F$ as the direct limit of the system 
	$(\{F_n\}_{n\in\bbN},\{i_{n,m}\}_{n,m\in\bbN})$ where 
	the ordering is given by divisibility. $P$, $Q$, and 
	$U$ are the corresponding images of $\{P_n\}$, $\{Q_n\}$, 
	and $\{U_n\}$. We define $q$ as the image of $q_1$ in $F$.

	The action of the matrix $T$ on $\cO_\cH$ yields a
	right action of the matrix $T$ on $F$ which is 
	characterized by acting trivially on coefficients 
	and mapping $q_n$ to $\zeta^{1/n}q_n$.
\end{defn}

\begin{prop}\label{propthetaunitdist}
	Let $\sigma:(\bbQ^2/\bbZ^2)\to U$ be the function defined such that
	\[
		\sigma((x_1,x_2)+\bbZ^2)=\prod_{y\in \bbQ^+\cap (x_1+\bbZ)}(1-q^y\zeta^{x_2})
		\times \prod_{y\in \bbQ^-\cap (x_1+\bbZ)}(1-q^{-y}\zeta^{-x_2}).
	\]
	$\sigma$ satisfies the condition of Proposition~\ref{propdistconstruction}.
	We call the corresponding distribution $u\in \cD'(V,U)$
	the theta unit distribution.
\end{prop}

\begin{proof}
	We need to show that for all $n\in \bbN$ and 
	$x_1,x_2\in \bbQ^2- (\tfrac{1}{n}\bbZ)^2$, 
	we have
	\[
		\sigma((nx_1,nx_2)+\bbZ^2) 
		= \prod_{i,j=0}^{n-1} \sigma\left(\left(x_1+\tfrac{i}{n},x_2
		+\tfrac{j}{n}\right)+\bbZ^2\right).
	\]
	Lemma~\ref{lemcycpoly} tells us that for all $\alpha,\beta\in \bbQ$,
	\[
		\prod_{j=0}^{n-1} (1-q^\alpha \zeta^\beta\zeta^{j/n})
		=1-q^{n\alpha}\zeta^{n\beta}.
	\]
	Fixing an $i\in \{0,\ldots,n-1\}$, we have the product
	\begin{align*}
		\prod_{j=0}^{n-1}\sigma((x_1+\tfrac{i}{n},x_2+\tfrac{j}{n})+\bbZ^2)
		&=\prod_{j=0}^{n-1}
		\prod_{\substack{y\in \bbQ^+\\ y\equiv a+\tfrac{i}{n} \pmod{\bbZ}}}
		(1-q^{y}\zeta^{x_2+j/n})\times \prod_{\substack{y\in \bbQ^-\\ 
		y\equiv a+\tfrac{i}{n} \pmod{\bbZ}}}(1-q^{-y}\zeta^{-x_2-j/n})\\
		&=\prod_{\substack{y\in \bbQ^+\\ y\equiv a+\tfrac{i}{n} \pmod{\bbZ}}}
		(1-q^{ny}\zeta^{nx_2})\times \prod_{\substack{y\in \bbQ^-\\ 
		y\equiv a+\tfrac{i}{n} \pmod{\bbZ}}}(1-q^{-ny}\zeta^{-nx_2})\\
		&=\prod_{\substack{y\in \bbQ^+\\ y\equiv na+i \pmod{n\bbZ}}}
		(1-q^{y}\zeta^{nx_2})\times \prod_{\substack{y\in \bbQ^-\\ 
		y\equiv na+i \pmod{n\bbZ}}}(1-q^{-y}\zeta^{-nx_2}).
	\end{align*}
	Taking the product over $i$ then results in the infinite products
	being taken over all $y\equiv na\pmod{\bbZ}$.
\end{proof}

\begin{prop}\label{propthetaTinvar}
	$u$ is invariant under the left action of $T$, i.e. for all test
	function $f$, $u(f\mid T)=u(f)\mid T$. 
\end{prop}

\begin{proof}
	Since $u$ is $\bbQ^\times$-invariant, it is sufficient to check
	this on test function in $f\in \cS_{\bbZ^2}(V')$, but this case
	is immediate from the definition of $\sigma$. 
\end{proof}

\subsection{Kato-Siegel Distribution}

We start with the Kato-Siegel units as defined in Scholl's article, which
we reiterate here.

\begin{thm}(\cite{Scholl},Theorem 1.2.1)\label{thm:ksunits}
	Take an integer $c$ such that $(c,6)=1$. There exists a unique
	rule $\vartheta_c$ which associates to each elliptic curve $E\to S$ over
	an arbitrary base $S$ a section $\vartheta_c^{(E/S)}\in \cO^\times(E-
	\Ker([\times c]))$ such that
	\begin{enumerate}
		\item as a rational function on $E$, $\vartheta_c^{(E/S)}$ has
			divisor $c^2(e)-\Ker([\times c])$, where $e$ is the
			identity of $E/S$;
		\item if $S'\to S$ is any morphism and $g: E'=E\times_S S'\to E$
			is the basechange, then $g^\ast \vartheta_c^{(E/S)}
			=\vartheta_c^{(E'/S')}$;
		\item if $\alpha:E\to E'$ is an isogeny of elliptic curves over
			a connected base $S$ whose degree is prime to $c$, then
			$\alpha_\ast \vartheta_c^{(E/S)}=\vartheta_c^{(E'/S)}$;
		\item $\vartheta_{-c}=\vartheta_c$, $\vartheta_1=1$, and if $c=m\cdot d$
			with $m,d\ge 1$, then $[\times m]_\ast \vartheta_c=
			\vartheta_c^{m^2}$ and $\vartheta_d\circ [\times m]=\vartheta_c
			/\vartheta_m^{d^2}$ (in particular, $[\times c]_\ast
			\vartheta_c = 1$);
		\item if $\tau\in \cH$ and $E_\tau/\bbC$ is the elliptic curve
			whose points are $\bbC/\bbZ+\tau\bbZ$, then 
			$\vartheta_c^{(E_\tau/\bbC)}$ is the function
			\[
				(-1)^{\frac{c-1}{2}}\Theta(u,\tau)^{c^2}
				\Theta(c u,\tau)^{-1}
			\]
			where 
			\[
				\Theta(u,\tau)=q^{1/12}(t^{1/2}-t^{-1/2})
				\prod_{n>0}(1-q^n t)(1-q^nt^{-1})
			\]
			and $q=\exp(2\pi i \tau)$, $t=\exp(2\pi i u)$.
	\end{enumerate}
\end{thm}

Since $\vartheta_1=1$, we fix an integer $c>1$ such that $(c,6)=1$ and let
$\sS=\{\ell: \ell\text{ prime such that }\ell\mid c\}$.

\begin{cor}\label{corksprops}
	Let $x=(a,b)+\bbZ^2\in (\bbZ_\sS^2/\bbZ^2)'$ and let
	$f=[x]\in \cS_{\bbZ^2}(V_\sS')$. Define a Kato-Siegel unit associated
	to $f$ by letting $\KSunit{c}{f}(\tau)\in \cO_\cH^\times$ be the function
	\[
		\KSunit{c}{f}(\tau)=\vartheta_c(a\tau+b,\tau).
	\]
	Then, for all $\gamma\in \SL_2(\bbZ)$ and all $\gamma\in T_\sS$ with
	integral coefficients, we have
	\[
		\KSunit{c}{f\mid \gamma}(\tau)
		=\KSunit{c}{f}(\gamma\tau).
	\]
\end{cor}

\begin{proof}
	The case where $\gamma\in\SL_2(\bbZ)$ is precisely
	Lemma~1.7 of \cite{Kato} by noting that our
	$\KSunit{c}{f}(\tau)$ is the unit defined as
	$\KSunit{c}{a,b}(\tau)$. 
	The case where $\gamma\in T_\sS$ is follows
	from Lemma~2.12 of \cite{Kato} as written in
	Proposition~2.2.2 of \cite{LLZ}. (See also
	Remark~2.2.3 of \cite{LLZ}.)
\end{proof}

By Corollary~\ref{cordistconstruction} and Corollary~\ref{corksprops}
restricted to the case when $\gamma$ is a scalar matrix, there exists a 
measure in $\cD'(\Vadele{\sS},\cO_{\cH}^\times)$ corresponding to
the Kato-Siegel units of Corollary~\ref{corksprops}. In fact, 
Lemma~\ref{lem:gl2gens} tells us that this distribution is 
$G_\sS$-invariant. 

\begin{defn}\label{defnkatosiegel}
	The Kato-Siegel distribution $\mu_{KS}\in H^0(G_\sS,
	\cD'(\Vadele{\sS},\cO_\cH^\times))$ is the unique distribution 
	corresponding to the assignment of Corollary~\ref{corksprops}. 
\end{defn}

\section{Dedekind-Rademacher Measures and Darmon-Dasgupta Measures}

In this section, we turn to group cohomology. First, we define a
method of smoothing cohomology class by an element of $\bbZ[G]$.
Then, we review the construction of the Dedekind-Rademacher
measures in \cite{DPVdr} and define its $\delta$-smoothed counterpart. 
Lastly, we will recall the definition of the Darmon-Dasgupta measures which 
we formulate as a cocycle to finally write down an explicit comparison
between the two constructions.

\subsection{Smoothing Cohomology}
We first note that the following process is a valid construction for
any group $G$. We will also define all maps on the level of cochains
to facilitate calculations in \S3.
Recall that for a subgroup $H\subset G$ and $\delta=\sum n_g\cdot [g]
\in \bbZ[G]$, we defined $H_\delta\subset H$ as
\[
	H_\delta=\left(\bigcap_{g\in G;n_g\neq 0} gHg^{-1}\right)\cap H.
\]
Letting $M$ be a left $G$-module, the map 
\begin{align*}
	(-)^\delta:M&\to M\\
	m&\mapsto m^\delta:= \delta\ast m
\end{align*}
restricts to a group homomorphism $M^H\to M^{H_\delta}$.

Now, let $M$ and $N$ be left $G$-modules with an $H$-homomorphism
$\eta:N\to M$. This homomorphism induces a map $\eta:N^G \to M^H$.

Let $F^\bullet(G,M)$ be the complex defined by the following data:
\begin{itemize}[topsep=0pt]
	\item $F^n(G,M):= \{\text{Functions }f:G^{n+1}\to M\}$
	\item $d^n:F^n(G,M)\to F^{n+1}(G,M)$ where
		\[f\mapsto (d^n(f):(g_0,\ldots,g_{n+1})
		\mapsto \sum_{i=0}^{n+1} (-1)^i f(g_0,\ldots,\hat{g_i},\ldots,g_{n+1})),\]
\end{itemize}
where $(g_0,\ldots,\hat{g_i},\ldots,g_{n+1})$ denotes the $n+1$-tuple
obtained by omitting the $i$-th entry. $F^\bullet(G,M)$ has a 
natural left $G$-action where for $\gamma\in G$, $\mathbf{g}\in G^{n+1}$, 
and $f\in F^n(G,M)$, we have 
\[
	(\gamma\ast f)(\mathbf{g})=\gamma\ast(f(\gamma^{-1}\mathbf{g})).
\]

The homogeneous cochain complex whose cohomology yields group cohomology
is obtained by taking $G$-invariants of $F^\bullet(G,M)$, i.e. the
complex $C^\bullet(G,M)$ where
\begin{align*}
	C^n(G,M)&:= F^n(G,M)^G\\
	d^n&:= d^n\mid_{C^n(G,M)}.
\end{align*}

Taking $\eta:N\to M$ as above, we obtain an $H$-module homomorphism
\begin{align*}
	\eta_\ast:F^\bullet(G,N)&\to F^\bullet(G,M)\\
	f&\mapsto \eta\circ f.
\end{align*}
One observes that this map commutes with the differentials 
of $F^\bullet(G,N)$ and $F^\bullet(G,M)$. As above, restricting
to the $G$-invariants yields a map
\[
	\eta_\ast:C^\bullet(G,N)\to F^\bullet(G,M)^H.
\]
Then, the action of $\delta$ on $F^\bullet(G,M)$ induces a map
\[
	(-)^\delta:F^\bullet(G,M)^H\to F^\bullet(G,M)^{H_\delta}.
\]
Last but not least, we can restrict functions on $G^{n+1}$
to functions on $H_\delta^{n+1}$ to obtain the restriction map
\[
	\Res:F^\bullet(G,M)^{H_\delta}\to C^\bullet(H_\delta,M).
\]
Putting this all together, we define the following map on
cochain complexes.

\begin{defn}\label{defncohsmoothing}
	Let $\delta\in \bbZ[G]$ and let $H\subset G$. Define
	$H_\delta\subset H$ as above and let $M$ and $N$ be left
	$G$-modules equipped with a group homomorphism $\eta:N\to M$ 
	that is $H$-invariant. We define
	$(\delta,\eta)^\bullet: C^\bullet(G,N) \to C^\bullet(H_\delta,M)$
	as the compositum 
	\[
		(\delta,\eta):=\Res\circ (-)^\delta\circ \eta_\ast.
	\]
\end{defn}

Since each of the maps used in defining $(\delta,\eta)$ commutes with
the differentials on their respective complexes, the following
is true.

\begin{prop}
	$(\delta,\eta)$ induces a map on cohomology groups
	\[
		(\delta,\eta)^\bullet: H^\bullet(G,N)\to H^\bullet(H_\delta,M).
	\]
\end{prop}

\subsection{Construction of the Dedekind-Rademacher cohomology class}
For the remainder of the paper, fix the following data:
\begin{itemize}
	\item $c>1$ an integer coprime to $6$,
	\item $N>1$ an integer coprime to $c$,
	\item $p\in \bbZ$ a prime that does not divide $Nc$,
	\item $\delta=\sum_{D\mid N, D>0} n_D\cdot [LR(D)] \in \bbZ[G]$
		satisfying
		$\sum_{D\mid N, D>0} n_D\cdot D=0$,
	\item $\sS$ is the set of primes dividing $c$,
		$\sN$ is the set of primes dividing $cN$,
		and $\sP$ is the set of all primes excluding
		$p$.
\end{itemize}

Let $\bbZ$ be the trivial right $G_\sS$-module, $\cO_\cH$ be the
right $G_\sS$-module of holomorphic functions on the complex
upper half plane $\cH$, and $\cO_\cH^\times$ be the right $G_\sS$-module
of non-vanishing holomorphic functions on $\cH$. We have the
following exact sequence of $G_\sS$-modules:
\[
	0\to \bbZ\to \cO_\cH \to\cO_\cH^\times\to 1,
\]
where $\cO_\cH\to \cO_\cH^\times$ is the exponential map $f\mapsto e^{2\pi i f}$.

By Proposition~\ref{proptestfree}, the functor $\cD'(\Vadele{\sS},-)$ is
exact, so we have a resulting exact sequence of distributions as left $G_\sS$-modules:
\begin{align}
	0\to \cD'(\Vadele{\sS},\bbZ)\to \cD'(\Vadele{\sS},\cO_\cH)\to 
	\cD'(\Vadele{\sS},\cO_\cH^\times)\to 1. \label{eqnfundses}
\end{align}
The associated long exact sequence of group cohomology gives us the
boundary map 
\[
	\partial_\sS:H^0(G_\sS,\cD'(\Vadele{\sS},\cO_\cH^\times))\to 
	H^1(G_\sS,\cD'(\Vadele{\sS},\bbZ)).
\]

\begin{defn}
	The Dedekind-Rademacher cohomology class $\mu_{DR}\in H^1(G_\sS,
	\cD'(\Vadele{\sS},\bbZ))$ is defined as $\mu_{DR} := \partial_\sS(\mu_{KS})$. 
\end{defn}

\begin{defn}
	The $\delta$-smoothed Dedekind-Rademacher cohomology class is 
	defined as $\mu_{DR}^\delta :=(\delta,(i_\sN^\sS)^\ast)^1(\mu_{DR})$.
\end{defn}

\begin{rem}
	For a right $G$-module $A$, we have an explicit description of 
	$(\delta,(i_\sN^\sS)^\ast)^0:H^0(G_\sS,\cD'(\Vadele{\sS},A))
	\to H^0(G_\sN(N),\cD(\Vadele{\sN},A))$. For $\mu\in H^0(G_\sS,\cD'(\Vadele{\sS},A))$,
	we have
	\begin{align*}
		(\delta,(i_\sN^\sS)^\ast)^0(\mu)(\gamma)
		&=\sum_{D\mid N} n_D LR(D)\ast((i_\sN^\sS)^\ast(LR(D^{-1})\gamma\ast\mu))\\
		&=\delta\ast((i_\sN^\sS)^\ast\mu).
	\end{align*}
\end{rem}

\subsection{The Darmon-Dasgupta Cocycle}

Let $\Gamma^p_0(N)\subset \SL_2(\bbZ[p^{-1}])$ be the congruence subgroup
\[
	\Gamma^p_0(N) := \left\{\gamma\in \SL_2(\bbZ[p^{-1}])\mid 
	\gamma\equiv \begin{bmatrix}
		\ast & \ast\\
		0&\ast
	\end{bmatrix}\pmod{N}\right\}.
\]
For all even $k\ge 2$, let $E_k(z)=-\tfrac{B_k}{2k}
+\sum_{n=1}^\infty \sigma_{k-1}(n)q^n$ denote the 
weight-$k$ Eisenstein series which is a modular form of 
level 1 and weight $k$ if $k\neq 2$. The case when $k=2$ is
not a modular form, though it is a limit of a real analytic
family of Eisenstein series.
Let $E_k^\delta(z):=-24\sum_{D\mid N}n_D D\cdot E_k(Dz)$ which
are Eisenstein series of level $\Gamma_0(N)$. 
Lastly, let $\bbX_0:=\bbZ_p^2 - p\bbZ_p^2$ be the set of 
primitive vectors.
$\bbX_0$ is a fundamental domain of the $D(p)$-action on $V_p'$.

We warn the reader that the following theorem's action of $\Gamma_0^p(N)$
on $V_p$ does not coincide with the rest of this paper. We will reconcile
this difference shortly.

\begin{thm}\cite{DD}\cite{Das}\label{thmDD}
	There exists a unique collection of $p$-adic measures on the space
	$V_p'$, indexed by pairs $(r,s)\in \Gamma_0^p(N)i\infty\times\Gamma_0^p(N)i\infty$
	denoted by $\tilde{\mu}_\delta\{r\to s\}$ satisfying the following properties:
	\begin{enumerate}
		\item For every homogeneous polynomial $h(x,y)\in \bbZ[x,y]$ of
			degree $k-2$,
			\[
				\int_{\bbX_0} h(x,y)\,d\tilde{\mu}_\delta\{r\to s\}(x,y)
				=\text{Re}\left((1-p^{k-2})\int_r^s h(z,1)E_k^\delta(z)\,dz\right)
			\]
		\item For all $\gamma\in \Gamma_0^p(N)$ and all compact, open $U\subset 
			V_p'$, 
			\[
				\tilde{\mu}_\delta\{\gamma r\to\gamma s\}(\gamma U)
				=\tilde{\mu}_\delta\{r\to s\}(U).
			\]
		\item For all compact, open $U\subset V_p'$,
			\[
				\tilde{\mu}_\delta\{r\to s\}(pU)=\tilde{\mu}_\delta\{r\to s\}(U).
			\]
	\end{enumerate}

	Furthermore, for all compact open sets of the form $U=
	\ivec{u\\ v}+p^n\bbZ_p^2\subset \bbX_0$
	and $\tfrac{\alpha}{\beta}\in \Gamma_0(N)\cdot i\infty$, 
	\[
		\tilde{\mu}_\delta\{i\infty\to \tfrac{\alpha}{\beta}\}(U)
		=-12\sum_{\ell=0}^{\beta-1} \bbB_1\left( \frac{\alpha}{\beta}
		\left(\ell+\frac{v}{p^n}\right)-\frac{u}{p^n} \right)\sum_{D\mid N}
		n_d\bbB_1\left( \frac{D}{\beta} \left(\ell+\frac{v}{p^n}\right)\right).
	\]
\end{thm}

\begin{rem}
	The explicit formula at the end of Theorem~\ref{thmDD} was computed 
	by Dasgupta in \cite{Das} while the rest of the theorem is from
	\cite{DD}. We note that the first criterion regarding the moments of
	$\tilde{\mu}_\delta$ is interchangeable with the
	explicit formula for compact open sets. To navigate between the two,
	one approximates characteristic functions of compact opens
	using polynomials to take the limit of these approximations using
	the first criterion. This is precisely the method carried out in \cite{Das}
	to derive the formula on compact opens in the first place.
	Conversely, one can calculate moments of the measure by taking
	refinements of $\bbX_0$ with smaller and smaller balls.
\end{rem}

\begin{rem}\label{remdivisor1}
	Assumption~2.1 of \cite{DD} further insists that $\delta$ be a degree 
	zero divisor, i.e. $\sum_D n_D=0$. The two conditions on $\delta$ 
	together makes the $\delta$-smoothed Ramanujan $\Delta$-function
	\[
		\Delta_\delta(\tau):= \prod_{D\mid N} \Delta(D\tau)^{n_D}
	\]
	into a modular unit that has no pole or zero at the infinity cusp.
	However, in \cite{DK} and \cite{FLcomp}, it is remarked that the 
	degree zero condition is unnecessary in carrying out 
	calculations. As we will see, this refinement is exactly
	the necessary condition for $\delta$-smoothing to wash out
	the nonparabolic aspect of $\mu_{KS}$ up to coboundaries.
	We further elaborate in Remark~\ref{remdivisor2}.
\end{rem}

As mentioned prior, $\tilde{\mu}_\delta$ is built considering $V_p$ as 
column vectors, which gives a natural left action of $\Gamma_0^p(N)$ that 
differs from our formulations. We reconcile this by letting $V_C$ be $\bbQ^2$
considered as column vectors and define $\phi:V\to V_C$ where $\phi(v)= Sv^T$, i.e. 
the transpose multiplied by the matrix
\[
	S=\begin{bmatrix}
		0&-1\\
		1&0
	\end{bmatrix}.
\]
We note that for all $\gamma\in G$, we have $S(\gamma^{-1})^T S^{-1}
=\gamma$, so taking the $S$-conjugate of the inverse transpose is an
automorphism of $G$. Taking the left action of $G$ on $V$ as defined 
in \S1, i.e. $\gamma\ast v=v\gamma^{-1}$,
$\phi$ is an isomorphism of left $G$-modules.
Since $\cS(\Vadele{\sP})=\cS(V_p)$, this defines a 
$G$-equivariant map $\phi:\cD'((V_C)_\sP,\bbZ)
\to \cD'(\Vadele{\sP},\bbZ)$. Notice that for all $\gamma\in \Gamma_0^p(N)$ and
$(a,b)+\bbZ_p^2\subset V_p'$, with $\alpha/\beta=\gamma i\infty$,
$\phi(\tilde{\mu}\{i\infty\to \gamma i\infty\})([U])$ is equal to
\[
	\phi(\tilde{\mu}_\delta\{i\infty\to \gamma i\infty\})([U])
		=-12\sum_{\ell=0}^{\beta-1} \bbB_1\left( \frac{\alpha}{\beta}
		(\ell+a)+b\right)\sum_{D\mid N}
		n_D\bbB_1\left( \frac{D}{\beta} (\ell+a)\right).
\]
Interpreting $\phi(\tilde{\mu}_\delta\{\gamma i\infty\to \gamma' i\infty\})$ 
as an element of $C^1(\Gamma_0^p(N),\cD'(\Vadele{\sP},\bbZ))$, we have 
the following proposition.

\begin{prop}
	There exists a cocycle $\mu_{DD,\delta}\in Z^1(G_{\sP}(N),
	\cD'(\Vadele{\sP},\bbZ))$ whose restriction to $\Gamma_0^p(N)$
	is $\phi\circ\tilde{\mu}_\delta$. This cocycle is unique up to
	coboundaries.
\end{prop}

\begin{proof}
	To see the uniqueness, we note that the determinant map yields the 
	short exact sequence
	\[
		1\to \Gamma_0^p(N)\to G_\sP(N)\to p^\bbZ\to 1.
	\]
	Letting $\cD:=\cD'(\Vadele{\sP},\bbZ)$, the inflation-restriction 
	sequence then yields
	\[
		0\to H^1(p^\bbZ,H^0(\Gamma_0^p(N),\cD))\to H^1(G_{\sP}(N),
		\cD)\to H^0(p^\bbZ,H^1(\Gamma_0^p(N),\cD))\to 
		H^2(p^\bbZ,H^0(\Gamma_0^p(N),\cD)).
	\]
	Since $p^\bbZ$ is a free group, $H^2(p^\bbZ,H^0(\Gamma_0^p(N),\cD))=0$.
	Furthermore, $\Gamma_0(N)$ acts transitively on $\bbX_0$,
	so any $\mu\in H^0(\Gamma_0^p(N),\cD)$ must be a uniform distribution.
	This implies that for all positive integers $n$, 
	$\mu([(0,1)+p^n\bbZ_p^2])=\tfrac{1}{p^{2n-2}}\cdot\mu([(0,1)+p\bbZ_p^2])$.
	Since $\mu$ must have integer coefficients, $\mu$ must be $0$, meaning
	$H^0(\Gamma_0^p(N),\cD)=0$. Thus, the restriction map is an isomorphism.

	Using the homogeneous resolution for our group complex, we define
	a function $C:G_{\sP}^2\to \cD'(\Vadele{\sP},\bbZ)$ where for all
	$\gamma_1,\gamma_2\in G_{\sP}=\GL_2^+(\bbZ[p^{-1}])$, 
	\[
		C(\gamma_1,\gamma_2)(f)=\tilde{\mu}_\delta\{\gamma_1i\infty
		\to \gamma_2i\infty\}(\phi^\ast f).
	\]
	$\tilde{\mu}_\delta$ being a partial modular symbol provides
	the cocycle relation: for all $\gamma_1,\gamma_2,
	\gamma_3\in G_\sP$, 
	\[
		C(\gamma_1,\gamma_2)+C(\gamma_2,\gamma_3)=C(\gamma_1,\gamma_3).
	\]

	We now check that $C$ is $G_{\sP}$-invariant. 
	We note that Property 2 and 3 of Theorem~\ref{thmDD} already 
	tells us that $C$ is $\Gamma_0^p(N)$ and $D(p)$ invariant.
	By the cocyle relations and $D(p)$-invariance, it suffices to
	check the following identity for $U=(a,b)+\bbZ_p^2\subset V_p'$:
	\[
		C(Id,\gamma)([U]\mid UL(p))=C(UL(p),UL(p)\gamma)([U]).
	\]

	In this case, we have
	\[
		[U]\mid UL(p)=[U \cdot UL(p)]=\sum_{i=0}^{p-1}
		\left[(a+\bbZ_p)\times \left(\frac{b+i}{p}+\bbZ_p\right)\right]
		\mid D(p).
	\]
	Letting $\gamma i\infty = \alpha/\beta$,
	\begin{align*}
		C(Id,\gamma)([U]\mid UL(p))
		&=\sum_{i=0}^{p-1}C(Id,\gamma)([(a,(b+i)/p)+\bbZ_p^2])\\
		&= \sum_{i=0}^{p-1}-12\sum_{\ell=0}^{\beta-1}
		\bbB_1 \left(\frac{\alpha}{\beta}(\ell+a)+\frac{b}{p}+\frac{i}{p}\right)
		\sum_{D\mid N} n_D \bbB_1\left(\frac{D}{\beta}(\ell+a)\right)\\
		&= -12\sum_{\ell=0}^{\beta-1}
		\left(\sum_{i=0}^{p-1}\bbB_1 \left(\frac{\alpha}{\beta}(\ell+a)
		+\frac{b}{p}+\frac{i}{p}\right)\right)
		\sum_{D\mid N} n_D \bbB_1\left(\frac{D}{\beta}(\ell+a)\right)\\
		&= -12\sum_{\ell=0}^{\beta-1}
		\bbB_1 \left(\frac{p\alpha}{\beta}(\ell+a)+b\right)
		\sum_{D\mid N} n_D \bbB_1\left(\frac{D}{\beta}(\ell+a)\right)\\
		&= C(UL(p),UL(p)\gamma_2)([U]).
	\end{align*}
\end{proof}

\subsection{Main Theorem}
To make a comparison between $\mu_{DR}^\delta$ and $\mu_{DD,\delta}$,
we need to make a few final adjustments to balance the ingredients 
of these two objects. First off, $\mu_{DR}^\delta$ needs to be restricted
to a $p$-adic distribution, so we define the cohomology class
$\mu_{DR}^{p\delta}:=([Id],(i_\sP^\sN)^\ast)^1 \mu_{DR}^\delta$.
On the other hand, $\mu_{DD,\delta}$ makes no reference to the integer
$c$ that is used to construct $\mu_{DR}$ by way of $\mu_{KS}$, so
we define another cocycle $(\mu_{DD,\delta})_c:=(c^2[Id]-[D(c)],Id)^1
\mu_{DD,\delta}$. 

\begin{thm}\label{thmmaincomparison}
	With $c$, $N$, $p$, and $\delta$ as delineated at the beginning
	of this section, $-(\mu_{DD,\delta})_c$ belongs to the cohomology
	class $12\cdot \mu^{p\delta}_{DR}$.
\end{thm}
\section{Computation of Cocycles}
We now compute cocycle representatives of $\mu_{DR}$
and $\mu_{DR}^\delta$. We will continue using the
homogeneous complex in computing cohomology.

We recall the Snake Lemma which will serve as the roadmap
for the following calculations. We start with the short exact sequence
\[
	0\to \cD'(\Vadele{\sS},\bbZ)\to \cD'(\Vadele{\sS},\cO_\cH)\to 
	\cD'(\Vadele{\sS},\cO_\cH^\times)\to 1.
\]
For $\mu_{KS}\in H^0(G_\sS,\cD'(\Vadele{\sS},\cO_\cH^\times))$,
we take $\tilde{E}\in C^0(G_\sS,\cD'(\Vadele{\sS},\cO_\cH))$ such that 
$e^{2\pi i \tilde{E}}=\mu_{KS}$ (which we will define as a concrete lift in \S3.2).
Then, $d^0(\tilde{E})\in C^1(G_\sS,\cD'(\Vadele{\sS},\cO_\cH))$ is in the kernel of 
the exponential map, so $d^0(\tilde{E})$ lies in the image of 
$C^1(G_\sS,\cD'(\Vadele{\sS},\bbZ))$. $d^0(\tilde{E})$ is a 
cocycle representative of $\mu_{DR} := \partial_\sS(\mu_{KS})$.

For this purpose, we pursue the following goals in this section:
\begin{enumerate}[topsep=0pt]
	\item Describe the Kato-Siegel distribution as a product of 
		simpler distributions (\S~\ref{distex}) to define a 
		section $\tilde{E}$;
	\item Review classical calculations of periods of 
		Eisenstein series;
	\item Calculate values of $\DedeRa^\delta:=d^0\tilde{E}^\delta$ 
		to make a direct comparison with the Darmon-Dasgupta 
		cocycle;
	\item Calculate characterizing values of $\DedeRa_\sS:=
		d^0\tilde{E}$.
\end{enumerate}

\subsection{Decomposition of Kato-Siegel Units}

Recall the Theta function $\Theta:\bbC\times \cH\to \bbC$ from 
Theorem~\ref{thm:ksunits} where for all $u\in \bbC$, $\tau\in \cH$, 
$t:=e^{2\pi i u}$, and $q:=e^{2\pi i \tau}$,  
\[
	\Theta(u,\tau):=q^{1/12}(t^{1/2}-t^{-1/2})\prod_{n>0} (1-q^nt)(1-q^nt^{-1}).
\]
Define $\theta:\bbQ^2\to\cO_{\cH}$ as the function which
assigns to $(a,b)\in \bbQ^2$ the holomorphic function
\[
	\theta(a,b)(\tau)=\Theta(a\tau+b,\tau)
\]

In Theorem~\ref{thm:ksunits}, $\Theta(u,\tau)$ was used
to define the Kato-Siegel units by assigning to each coset
$v=(a,b)+\bbZ^2\in (\bbZ_{\sS}^2/\bbZ^2)'$ the modular unit
\[
	\KSunit{c}{[v]}(\tau)=(-1)^{(c-1)/2}
	\frac{\theta(a,b)(\tau)^{c^2}}{\theta(ca,cb)(\tau)}.
\]
We will now simplify the $q$-expansion of $\theta(a,b)$.

For simplicity, we represent $v$ with $(a,b)$
non-negative rational numbers since $\KSunit{c}{[v]}$ 
only depends on the coset $v$ and not its representative $(a,b)$.
We also recall that for any rational number $b$, $\zeta^b$ 
denotes the complex number $e^{2\pi i b}$.

\subsubsection{Simplification of $\theta(a,b)$}
We begin with the expansion of $\theta(a,b)$ as
\[
	q^{1/12}(q^{a/2}\zeta^{b/2}-q^{-a/2}\zeta^{-b/2})
	\prod_{n>0}(1-q^{n+a}\zeta^b)(1-q^{n-a}\zeta^{-b}).
\]
By Definition~\ref{defnqexp}, we can realize $\theta(a,b)$
as an element of $P$. Since $P$ lives in the local ring
of power series, we can express the image of $\theta$ 
in $F^\times\cong Q\times U$. In other words,
we can express $\theta(a,b)$ as a product of the form
\[
	q^A B \prod_{x>0}(1-q^xC_x)
\]
which will give a similar factorization of $\mu_{KS}$.

$a$ being non-negative means we only need to manipulate the 
following factors of $\theta(a,b)$.
\begin{align}
	q^{a/2}\zeta^{b/2}-q^{-a/2}\zeta^{-b/2}
		&=q^{-a/2}(-\zeta^{-b/2})(1-q^a\zeta^b) \label{thetafactor1}\\
	\prod_{n=1}^{\floor{a}}1-q^{n-a}\zeta^{-b}
		&=\prod_{n=1}^{\floor{a}}-q^{n-a}\zeta^{-b}
		(1-q^{a-n}\zeta^b) \label{thetafactor2}\\
		(\ref{thetafactor1})\times (\ref{thetafactor2})
		&=q^{\frac{1}{2}(-a+\floor{a}(\floor{a}+1-2a))}
		\zeta^{\frac{1}{2}(1-b+\floor{a}(1-2b))}\nonumber\\
		&\times\DCycdist([(a,b)+\bbZ^2])\prod_{n=0}^{\floor{a}}(1-q^{n+\langle a\rangle}\zeta^b).\nonumber
\end{align}

Isolating the factors of $q$-power yields:
\begin{align*}
	\Ord_q\theta(a,b)=\frac{1}{2}\left( \frac{1}{6}-a+\floor{a}(\floor{a}+1-2a) \right)
\end{align*}
A direct calculation shows that
\[
	\Ord_q\theta(a+1,b)=-a-\frac{1}{2}+\Ord_q\theta(a,b).
\]
The expression $\tfrac{1}{2}(\bbB_2([a+\bbZ])-a^2)$ shares this transformation
property and coincides with $\Ord_q\theta(a,b)$ when $0\le a<1$, so letting
$f=[(a,b)+\bbZ^2]$, we have
\[
	\Ord_q\theta(a,b)=\frac{1}{2}((\bbB_2\otimes\bbB_0)(f)-a^2).
\]

Recalling the Theta-unit distribution $u$ from \S~\ref{distex}, we 
are left with the following proposition.
\begin{prop}\label{propthetadecomp}
	For all $(a,b)\in \bbQ^2$ such that $a,b\ge 0$ and $f=[(a,b)+\bbZ^2]$, we 
	have the equality
	\[
		\theta(a,b)(\tau)=q^{\frac{1}{2}((\bbB_2\otimes\bbB_0)(f)-a^2)}\cdot
		\zeta^{-B_1(b)\cdot (1+\floor{a})+\frac{b}{2}}\cdot\DCycdist(f)\times
		u(f)(\tau)
	\]
\end{prop}

\subsubsection{Simplification of $\mu_{KS}$}
With the benefit of hindsight, we start with a definition.
\begin{defn}
	$\pi_c\in \cD'(V_\sS,\bbQ)$ is defined as the 
	$\bbQ^\times$-invariant distribution which maps $[(a,b)+\bbZ^2]
	\in \cS_{\bbZ^2}(V_\sS')$ to $B_1(\FPart{bc})(cB_1(\FPart{a})-B_1(\FPart{ac}))$.
	Note that we can write $\pi_c$ as a difference of tensors of
	the distributions $B_1$ and the pullback of $B_1$ with respect
	to the action of $c$ on $\cD(\bbQ_\sS,\bbQ_0)$. This makes
	$\pi_c$ invariant under the action of $T_\sS$.
\end{defn}

Returning our attention to the Kato-Siegel units, 
Proposition~\ref{propthetadecomp} tells us that the first non-zero 
coefficient of the $q$-expansion of $\KSunit{c}{[(a,b)+\bbZ^2]}$ is 
\[
	(-1)^{(c-1)/2}\cdot\DCycdist_c([(a,b)+\bbZ^2])\cdot 
	\zeta^{-c^2B_1(b)\cdot (1+\floor{a})+\frac{bc^2}{2}
	+B_1(bc)\cdot (1+\floor{ac})-\frac{bc}{2}}
\]
Writing $(-1)^{(c-1)/2}$ as $\zeta^{(c-1)/4}$ and setting aside the
$\DCycdist$-factor for now, we are left with $\zeta$ raised to the
following power.
\begin{align}
	\frac{c-1}{4}-&c^2B_1(b)\cdot (1+\floor{a})+\frac{bc^2}{2}+B_1(bc)\cdot (1+\floor{ac})-\frac{bc}{2}\nonumber\\
	&=\frac{c-1}{2} -c^2B_1(b)(1+\floor{a})+B_1(bc)(1+\floor{ac}+(c-1)/2)\nonumber\\
	&\equiv -c^2B_1(b)(1+\floor{a})+cB_1(bc)+B_1(bc)(\floor{ac}-(c-1)/2)\pmod{\bbZ}\nonumber\\
	&\equiv -c^2B_1(b)\floor{a}+B_1(bc)(\floor{ac}-(c-1)/2)\pmod{\bbZ}.\label{eqnnzcoeff}
\end{align}
As expected, we see that this value is invariant modulo $\bbZ$ when $a$ or $b$ is
translated by $1$. However, translating $b$ by $1/c$ also leaves the value invariant
modulo $\bbZ$. $\pi_c$ shares the same invariance property modulo $\bbZ$. Furthermore, 
letting $a=x+\tfrac{j}{c}$ and $b$ for $x,b\in [0,\tfrac{1}{c})$ and $j\in \{0,\ldots,c-1\}$
we see that 
\[
	(\ref{eqnnzcoeff})=B_1(bc)(j-(c-1)/2)=\pi_c([(a,b)+\bbZ^2]).
\]

Consequently, for all $f\in \cS(V_\sS')$, the first non-zero coefficient of $\mu_{KS}(f)$ is
\begin{align}
	\DCycdist_c(f)\cdot 
	\zeta^{\pi_c(f)}.
\end{align}
This leaves us with the following simple factorization of $\mu_{KS}$.
\begin{prop}\label{propksunitdecomp}
	Let $(\bbB_2\otimes \bbB_0)_c$, $\DCycdist_c$, and $u_c$ denote
	the respective distributions smoothed by the group algebra 
	element $c^2[Id]-[D(c)]$. Then, we have the following 
	decomposition:
	\[
		\mu_{KS}=
		q^{\frac{1}{2}(\bbB_2\otimes\bbB_0)_c}\cdot \left(
		\zeta^{\pi_c}
		\cdot \DCycdist_c\right)\cdot
		u_c.
	\]
\end{prop}

\subsection{Cocycle Representatives of $\mu_{DR}$ and $\mu_{DR}^\delta$}

\begin{defn}\label{defnkssection}
	$\rho,\tilde{E}\in \cD'(\Vadele{\sS},\cO_\cH)$ are the distributions
	which  map $f\in \cS(V_\sS')$ and $\tau\in \cH$ as follows.
	\begin{align*}
		\rho(f)(\tau) &:= \frac{\tau}{2}(\bbB_2\otimes \bbB_0)_c(f)
		+\pi_c(f)\\
		\tilde{E}(f)(\tau) &:= \rho(f)(\tau)+\frac{1}{2\pi i} 
		(\delta_0\otimes \Cyclog)_c(f)
		+\frac{1}{2\pi i} \int_{i\infty}^\tau \dlog(u_c(f)).
	\end{align*}
	Note that $\sS$ is implicit in the definition of $\tilde{E}$ via the integer $c$.
	Abusing notation, we also denote by $\tilde{E}$ the 0-cochain $G_\sS\to \cD'(V_\sS,
	\cO_\cH)$ satisfying $\gamma\mapsto \gamma\ast\tilde{E}$. 
\end{defn}

\begin{rem}\label{remdlogu}
	The key observation for the remainder of this paper is 
	that $\dlog(u(f))$ is precisely a weight 2 Eisenstein 
	series minus the constant term of its Fourier expansion 
	at $i\infty$. We can see this explicitly by observing 
	that for all $[(a,b)+\bbZ^2]\in\cS(V')$, we have the 
	identity $\dlog(1-q^a\zeta^b)=-2\pi ia\sum_{k=1}^\infty 
	(q^a\zeta^b)^k$. Since $u([(a,b)+\bbZ^2])$ is an 
	absolutely convergent infinite product on $\cH$, we 
	have \[
		\frac{1}{2\pi i}\dlog(u[(a,b)+\bbZ^2])
		=(\phi_{(a,b)}-\frac{1}{2}\bbB_2(a))\dtau
	\]
	where $\phi_{(a,b)}$ is the distribution of weight 2 
	Eisenstein series defined in page~57 of \cite{Stevens}.

	It follows by Lemma~2.1.1 and the following discussion
	on page~46 of \cite{Stevens} that the integral
	\[
		\int_{i\infty}^\tau\dlog(u_c(f))
	\]
	converges and defines a function in $\cO_\cH$.
\end{rem}

\begin{rem}
	By Proposition~\ref{propksunitdecomp}, $u_c$ takes values
	in $\cO_\cH^\times$, so $\dlog(u_c(f))$ is a well-defined
	differential on $\cH$ for all $f\in \cS(V_\sS')$.
\end{rem}

\begin{rem}
	Since each term of $\tilde{E}$ is invariant under the action by
	scalar matrices in $G_\sS$, we can reduce all calculations from this
	point onward to $\bbZ^2$-invariant test functions.
\end{rem}

\begin{defn}
	$\DedeRa_\sS\in Z^1(G_\sS,\cD(\Vadele{\sS},\bbZ))$
	is the cocycle $\DedeRa_\sS = d^0\tilde{E}$. 
\end{defn}

By Proposition~\ref{propksunitdecomp}, $\DedeRa_\sS$ is a cocycle representative
of $\mu_{DR}$. We are interested in computing the $\delta$-smoothed counterpart
of $\DedeRa_\sS$ where $\delta=\sum_{D\mid N} n_D\cdot [LR(D)]\in\bbZ[G]$ such that 
$\sum_{D\mid N} n_D\cdot D=0$. 
of Defintion~\ref{defncohsmoothing}.

\begin{defn}
	Recalling the smoothing map $(\delta,(i_\sN^\sS)^\ast)$
	of Definition~\ref{defncohsmoothing}, the $\delta$-smoothed 
	Kato-Siegel distribution is $\mu_{KS}^\delta:=
	(\delta,(i_\sN^\sS)^\ast)^0 (\mu_{KS})\in 
	\cD'(V_\sN,\cO_\cH^\times)$. 
\end{defn}

\begin{notn}
	For $f\in \cS((\Vadele{\sN})')$ and $D\mid N$, we define $f^D\in \cS((\Vadele{\sS})')$
	as $i_\sN^\sS(f\mid LR(D))$. We extend this to test functions in $\cS(V_\sN')$
	via the map $\phi$, i.e. $f^D\in \cS(V_\sS')$ is the test
	function $\phi(i_\sN^\sS (\phi^{-1}(f)\mid LR(D)))$. If $f=[(a,b)+\bbZ^2]\in
	\cS(V_\sN')$, $f^D=[(a,bD)+\bbZ^2]$.
\end{notn}

\begin{prop}\label{propsmoothedqorder}
	For all $f\in \cS(V_\sN')$, $\Ord_q(\mu_{KS}^\delta(f))=0$.
\end{prop}

\begin{proof}
	By Proposition~\ref{propksunitdecomp}, $\Ord_q(\mu_{KS}(f))
	=\frac{1}{2}(\bbB_2\otimes \bbB_0)_c(f)$. 
	Since for all $f\in \cS(V_\sN')$ and $D\mid N$, 
	\[
		(\bbB_2\otimes \bbB_0)(f^D)=(\bbB_2\otimes \bbB_0)(f),
	\]
	so noting that the action of $LR(D^{-1})$ on $\bbQ_1\otimes
	\bbQ_{-1}$ is multiplication by $D$, 
	\begin{align*}
		\tfrac{1}{2}\cdot\Ord_q(\mu_{KS}^\delta(f))
		=\sum_{D\mid N} n_D\cdot D\cdot 
		(\bbB_2\otimes \bbB_0)_c(f^D)
		=\sum_{D\mid N} n_D\cdot D\cdot 
		(\bbB_2\otimes \bbB_0)_c(f)=0.
	\end{align*}
\end{proof}

\begin{defn}
	We define $\tilde{E}^\delta:=(\delta,(i_\sN^\sS)^\ast)^0(\tilde{E})\in 
	\cD'(\Vadele{\sN},\cO_\cH)$. By Proposition~\ref{propsmoothedqorder}, 
	\begin{align*}
		\tilde{E}^\delta(f)(\tau)&=\sum_{D\mid N} n_D\cdot\tilde{E}(f^D)(D\tau)\\
		&=\sum_{D\mid N} n_D \cdot \left(\pi_c(f^D)
		+\frac{1}{2\pi i} (\delta_0\otimes \Cyclog)_c(f^D)+\frac{1}{2\pi i}
		\int_{i\infty}^\tau \dlog(u_c(f^D)(Dz))\right).
	\end{align*}
\end{defn}
We have $e^{2\pi i\tilde{E}^\delta}=\mu_{KS}^\delta$ since 
$(\delta,(i_\sN^\sS)^\ast)$ preserves sections. Thus, the following is a cocycle 
representative of the $\delta$-smoothed Dedekind-Rademacher cohomology class.

\begin{defn}
	We define $\DedeRa^\delta:=d^0\tilde{E}^\delta\in 
	Z^1(G_\sN(N),\cD'(\Vadele{\sN},\bbZ))$.
\end{defn}

\subsection{Periods of Eisenstein Series}

Let $r\in \SL_2(\bbZ)\cdot i\infty$ and $f\in \cS(V')$. We 
are interested in computing the following integral
\begin{align}
	\frac{1}{2\pi i}\int_r^{i\infty}\dlog(u(f)). \label{dloguint}
\end{align}

\begin{rem}
	As we are working with Eisenstein series, it is important
	to clarify the path of integration. All following integrals
	will have an endpoint at $i\infty$, so for a rational cusp
	$r$, we specify once and for all that $\int_{r}^{i\infty}$
	denotes the integral along the path parametrized by $it+r$ 
	where $t\in \bbR_{\ge 0}$.
\end{rem}

\begin{defn}
	Let $D$ be a positive integer and $(a,b)+\bbZ^2\in (\bbQ/\bbZ)^2$. 
	Define $E^+_D(a,b)(\tau)$ and $E^-_D(a,b)(\tau)$ by the $q$-expansions
	\[
		E^\pm_D(a,b)(\tau):=-2\pi i D\sum_{\substack{x\equiv \pm a\pmod{\bbZ}\\x>0}}
		x\sum_{m=1}^\infty (q^{Dx} \zeta^{\pm b})^m.
	\]
	We define $E_D(a,b)(\tau):= E^+_D(a,b)(\tau)+E^-_D(a,b)(\tau)$.
	By Remark~\ref{remdlogu}, we observe the identity 
	$\dlog(u([(a,b)+\bbZ^2])(D\tau)) = E_D(a,b)(\tau)\,\dtau$. 
\end{defn}

To compute the periods of $\dlog(u(f)(D\tau))$, we will make extensive
use of Mellin transforms of $E_D(a,b)(\tau)$ as defined below.

\begin{defn}
	Let $F\in\cO_\cH$ be a weight 2 modular form of some level
	$n$ with $q$-expansion $F(\tau)=\sum_{k=0}^\infty 
	a_k q_n^k$ and let $\tilde{F}$ be the function 
	$\tilde{F}(\tau)=F(\tau)-a_0$. 
	For $s\in \bbC$ and $y$ the imaginary part
	of $\tau$, we define the Mellin transform
	\[
		\sM(F,s)=\int_0^{i\infty} 
		\tilde{F}(\tau)\cdot y^{s-1}\,\dtau.
	\]
	By Propostion~2.1.2c of \cite{Stevens}, this integral
	is well-defined away from $s=0,2$.
\end{defn}

We now reduce the Mellin transform of $E_D(a,b)$ to a sum of 
Mellin transforms of $E_1$. This will allow us to use 
Proposition~2.5.1 of \cite{Stevens} to
take the limit of $\sM(E_1(a,b),s)$ as $s$ approaches $1$.

\subsubsection{Simplifications}
\begin{prop}\label{propremovelvl}
	Let $D>0$ be an integer. Then, for all $a,b\in\bbQ$ and $\alpha/\beta\in \bbQ$,
	\[
		\int_{\alpha/\beta}^{i\infty} E_D(a,b)(\tau)\cdot y^{s-1}\,\dtau
		=D^{1-s}\int_{D\alpha/\beta}^{i\infty} E_1(a,b)(\tau)\cdot y^{s-1}\,\dtau.
	\]
\end{prop}

\begin{proof} By a standard change of variables, we have
	\begin{align*}
		\int_{\alpha/\beta}^{i\infty} E_D(a,b)(\tau)\cdot y^{s-1}\,\dtau
		&= \int_{\alpha/\beta}^{i\infty} D\cdot E_1(a,b)(D\tau)\cdot y^{s-1}\,\dtau\\
		&= D^{1-s} \int_{D\alpha/\beta}^{i\infty} E_1(a,b)(\tau)\cdot y^{s-1}\,\dtau.
	\end{align*}
\end{proof}

We recall the following zeta functions for $x\in \bbQ/\bbZ$ and $s\in \bbC$ with real 
part greater than $1$.
\begin{align*}
	Z(s,x):= \sum_{n=1}^\infty \frac{e^{2\pi inx}}{n^s};\quad
	\zeta(s,x):= \sum_{\substack{k\equiv x\pmod{\bbZ}\\ k>0}}k^{-s}.
\end{align*}

\begin{prop}\label{propremovecusp}
	Let $\alpha,\beta\in \bbZ$ such that $\beta>0$ and $(\alpha,\beta)=1$.
	Let $y$ denote the imaginary part of $\tau$.
	For all $(a,b)\in (\bbQ/\bbZ)^2$, and $\ell=0,\ldots,\beta-1$, 
	define $R(a,\ell,\beta)=(a+\ell)/\beta$ and $Q(a,b,\ell,r)
	=b+r(a+\ell)$. Then, we have
	\begin{align*}
		\int_{\alpha/\beta}^{i\infty} E_1(a,b)(\tau)\cdot y^{s-1}\,\dtau
		=\beta^{1-s}\sum_{\ell=0}^{\beta-1} \sM(E_1(R(a,\ell,\beta),
		Q(a,b,\ell,\alpha/\beta)),s).
	\end{align*}
\end{prop}

\begin{proof}
	Consider the case when $s$ has real part greater than $2$ so that
	everything absolutely converges.
	We first consider the integral of $E_1^+(a,b)(\tau)$. The case of
	$E_1^-(a,b)$ will be entirely symmetric.
	\begin{align*}
		\int_{\alpha/\beta}^{i\infty} &E_1^+(a,b)(\tau)\cdot y^{s-1}\,\dtau\\
		&= -2\pi i \cdot\int_{\alpha/\beta}^{i\infty} 
		\sum_{\substack{x\equiv a\pmod{\bbZ}\\ x>0}}
		x\sum_{m=1}^\infty (q^{x}\zeta^b)^m y^{s-1}\,\dtau\\
		&= -2\pi i  \cdot \int_{0}^{i\infty} 
		\sum_{\substack{x\equiv a\pmod{\bbZ}\\ x>0}}
		x\sum_{m=1}^\infty (q^{x}\zeta^{b+x\alpha/\beta})^m y^{s-1}\,\dtau\\
		&= -2\pi i  \cdot \sum_{\substack{x\equiv a\pmod{\bbZ}\\ x>0}}
		x\sum_{m=1}^\infty \zeta^{m\left(b+x\alpha/\beta\right)}
		\cdot\int_{0}^{i\infty}q^{mx} y^{s-1}\,\dtau\\
		&= \frac{\Gamma(s)}{i(2\pi)^s}\cdot 
		\sum_{n=0}^\infty (a+n)^{1-s}\cdot \sum_{m=1}^\infty
		m^{-s}\zeta^{m\left(b+x\alpha/\beta\right)}.
	\end{align*}
	
	All series are absolutely convergent, so we decompose the 
	sum over $n$ by residue classes modulo $\beta$. 
	\begin{align*}
		\int_{\alpha/\beta}^{i\infty} &E_1^+(a,b)(\tau)\cdot y^{s-1}\,\dtau\\
		&= \frac{\Gamma(s)}{i(2\pi)^s}\cdot \sum_{\ell=0}^{\beta-1}
		\sum_{n=0}^\infty (\beta R(a,\ell,\beta)+\beta n)^{1-s}\sum_{m=1}^\infty 
		m^{-s}\zeta^{mQ(a,b,\ell,\alpha/\beta)}\\
		&= \frac{\Gamma(s)}{i(2\pi)^s}\cdot \beta^{1-s} 
		\sum_{\ell=0}^{\beta-1} \zeta(R(a,\ell,\beta),1-s)\cdot
		Z(Q(a,b,\ell,\alpha/\beta),s).
	\end{align*}
	The same computation holds for $E^-_1(a,b)(\tau)$. The proof
	is complete once we compare the summands with the case when
	$\alpha/\beta=0$.
\end{proof}

\begin{cor}\label{corremovecusp}
	With the same set-up as Proposition~\ref{propremovecusp} but with 
	a positive integer $D$, 
	\[
		\int_{\alpha/\beta}^{i\infty} E_D(a,b)(\tau)\cdot y^{s-1}\dtau
		= (D\beta)^{1-s} \sum_{\ell=0}^{\beta/D-1}
		\sM(E_1(R(a,\ell,\beta/D),Q(a,b,\ell,D\alpha/\beta)),s).
	\]
\end{cor}

\subsubsection{Periods of $\dlog(u(a,b))$}
Let $\gamma\in \SL_2(\bbZ)$ of the form
\[
	\gamma=\begin{bmatrix} \alpha & x\\ \beta&y\end{bmatrix}.
\]
For convenience, we insist that $\beta>0$ and define $r=\alpha/\beta$.
Notice that this implicitly ignores the trivial case when $r=i\infty$.
We now compute \ref{dloguint}.

By Corollary~\ref{corremovecusp}, we have the equality
\[
	\int_{r}^{i\infty} E_1(a,b)(\tau)\cdot y^{s-1}\,\dtau
	=\beta^{s-1} \sum_{\ell=0}^{\beta-1} 
	\sM(E_1(R(a,\ell,\beta),Q(a,b,\ell,\alpha/\beta),s).
\]

The final ingredient for the computation of periods is the following
proposition which provides the integral of $\dlog(u(f))$ along the 
imaginary axis.
\begin{prop}\label{thmdedekindsum} (\cite{Stevens} Proposition~2.5.1)
	For all $a,b\in \bbQ/\bbZ$ and $f=[(a,b)+\bbZ^2]$,  
	\[
		\tfrac{1}{2\pi i}\sM(E_1(a,b),1)
		=(\bbB_1\otimes \bbB_1)(f)
		-\frac{1}{2\pi i}
		(\delta_0\otimes \RCyclog)(f\mid S -f).
	\]
\end{prop}

We define the following Dedekind sums taking inspiration from \cite{Hal}
and Proposition~\ref{thmdedekindsum}.

\begin{defn}
	For $\alpha,\beta\in \bbZ$ with $\beta\neq 0$ and $a,b\in \bbQ$,
	let $C(\alpha,\beta,a,b)$ denote the Dedekind sum
	\[
		C(\alpha,\beta,a,b):=\sum_{\ell=0}^{\beta-1} \bbB_1(R(a,\ell,\beta))
		\bbB_1(Q(a,b,\ell,\alpha/\beta)).
	\]
\end{defn}

Using Proposition~\ref{thmdedekindsum} to take the limit as $s$ approaches $1$, 
we get:
\begin{align}
	\frac{1}{2\pi i}\int_{r}^{i\infty} &\dlog(u(a,b)(\tau))
	=C(\alpha,\beta,a,b)\nonumber\\
	&+\frac{1}{2\pi i}\sum_{\ell=0}^{\beta-1} \delta_0([R(a,\ell,\beta)+\bbZ])\cdot
	\log\abs{1-e^{2\pi i Q(a,b,\ell,\alpha/\beta)}} \label{deltaR}\\
	&-\frac{1}{2\pi i}\sum_{\ell=0}^{\beta-1} \delta_0([Q(a,b,\ell,\alpha/\beta)+\bbZ])\cdot
	\log\abs{1-e^{2\pi i R(a,\ell,\beta)}}. \label{deltaQ}
\end{align}

We now consider the terms (\ref{deltaR}) and (\ref{deltaQ}) to 
simplify the periods further. We will be considering the behaviors 
of these terms as $\ell$ varies while all other inputs are fixed, 
so we simplify notation by writing $R_{\ell}$ and $Q_{\ell}$ for 
$R(a,\ell,\beta)$ and $Q(a,b,\ell,\alpha/\beta)$, respectively. 
The key observation for the following will be the fact that 
$Q_{\ell}=b+\alpha R_\ell$.

(\ref{deltaR}): If $\delta_0([R_\ell+\bbZ])=1$, $Q_{\ell}\equiv 
b\pmod{\bbZ}$, so (\ref{deltaR}) can be rewritten as \[
	\sum_{\ell=0}^{\beta-1}\delta_0([R_\ell+\bbZ])\cdot\log\abs{1-e^{2\pi i b}}.
\]
If $a$ is not an integer, let $\fp$ be a prime such that the valuation
$v_\fp(a)<0$. Then, $v_\fp(a/\beta)<v_\fp(\ell/\beta)$ for all $\ell$,
so $R_\ell$ can never be an integer. Thus, $R_\ell\in 
\bbZ$ if and only if $a\in\bbZ$ and $\ell=0$, so we have \[
	(\ref{deltaR})=(\delta_0\otimes \RCyclog)([(a,b)+\bbZ]).
\]

(\ref{deltaQ}): Let $d\in \bbZ$ such that $da$ and $db$ are integers. Then,
\begin{align*}
	Q_\ell\in \bbZ
	&\Leftrightarrow -b\equiv \frac{\alpha}{\beta}(a+\ell)\pmod{\bbZ}\\
	&\Leftrightarrow -(a\alpha+b\beta)\equiv \alpha\ell\pmod{\beta\bbZ}\\
	&\Leftrightarrow -(da\alpha +db\beta)\equiv d\alpha\ell \pmod{d\beta\bbZ}.
\end{align*}
Thus, $Q_\ell\in \bbZ$ if and only if $\ell$ is the unique solution 
mod $\beta$ to the congruence $-(da\alpha+db\beta)\equiv d\alpha x
\pmod{d\beta\bbZ}$. Since $(\alpha,\beta)=1$, the existence of a 
solution is equivalent to $d\mid -(da\alpha+db\beta)$, i.e. 
$\delta_0([a\alpha+b\beta+\bbZ])$. Thus, there is at most one 
$\ell\in\{0,\ldots,\beta-1\}$ such that $\delta_0([Q_\ell+\bbZ])=1$ 
and such an $\ell$ exists if and only if $\delta_0([a\alpha+b\beta+\bbZ])=1$. 

Suppose now that $\delta_0([a\alpha+b\beta+\bbZ])=1$. We have the
matrix $\gamma\in\SL_2(\bbZ)$ with the first column $\ivec{\alpha\\\beta}$.
We have $\alpha \ell \equiv -(a\alpha+b\beta)\pmod{\beta\bbZ}$. Since 
$\alpha y\equiv 1\pmod{\beta}$, we have $\ell\equiv -y(a\alpha+b\beta)
\pmod{\beta}$. Applying this to $R_\ell$, we get
\begin{align*}
	R_\ell&\equiv \frac{a-a\alpha y -b\beta y}{\beta}\pmod{\bbZ}\\
	&= \frac{a-a(1+\beta x)-b\beta y}{\beta}\\
	&= \frac{-(a\beta x+b\beta y)}{\beta}\\
	&= -(ax+by).
\end{align*}
Finally, we note that $\RCyclog(q)=\RCyclog(-q)$ for all $q\in \bbQ$, 
so we conclude that 
\[
	(\ref{deltaQ})
	=(\delta_0\otimes \RCyclog)([(a,b)+\bbZ^2]\mid\gamma).
\]

Summarizing these computations, we are left with the following theorem.

\begin{thm}\label{thmeisenperiod}
	Let $f=[(a,b)+\bbZ^2]\in \cS(V')$ and let $\gamma\in \SL_2(\bbZ)$ with 
	$\gamma\cdot i\infty=\alpha/\beta$. Then,
	\begin{align*}
		\int_{\gamma i\infty}^{i\infty} \dlog(u(f))
		= C(\alpha,\beta,a,b)-
		\frac{1}{2\pi i}&\left((\delta_0\otimes \RCyclog)
		(f\mid \gamma-f)\right).
	\end{align*}
\end{thm}

\begin{rem}
	Having fixed $\alpha/\beta$, we have made a choice of $\gamma
	\in \SL_2(\bbZ)$. However, this choice is exactly up to a
	multiplication of a power of the matrix $T$ and this choice
	is washed out due to $\delta_0\otimes \RCyclog$ being invariant
	under the action of $T$.
\end{rem}

The following is a direct consequence of Corollary~\ref{corremovecusp}
and Theorem~\ref{thmeisenperiod}.

\begin{cor}\label{coreisenperiod}
	Let $D$ be a positive integer and take the the same assumptions 
	as Theorem~\ref{thmeisenperiod} with the added condition that
	$\gamma\in \Gamma_0(D)$. Let $\gamma_D\in \SL_2(\bbZ)$ such that
	$\gamma_D=UL(D)\cdot\gamma \cdot UL(D)^{-1}$. Then, we have the following formula.
	\[
		\frac{1}{2\pi i} \int_{\alpha/\beta}^{i\infty}
		\dlog(u(f)(D\tau))
		=C(\alpha,\beta/D,a,b)-\frac{1}{2\pi i} (\delta_0\otimes \RCyclog)(f\mid \gamma_D-f)
	\]
\end{cor}

\subsection{$\DedeRa^\delta$ Calculations}\label{subsectionexplicitdeltadr}

Before focusing on $\DedeRa^\delta$, we preemptively manipulate the formula
for $\mu_{DD,\delta}$.

\begin{prop}
	Let $f=[(a,b)+\bbZ^2]\in \cS(V_\sN')$.
	Then, for all $\alpha/\beta\in \Gamma_0(N)\cdot i\infty$ with $\beta> 0$,
	\[
		\mu_{DD,\delta}\{\infty\to \tfrac{\alpha}{\beta}\}(f)
		=\sum_{D\mid N} n_D\cdot C(\alpha,\beta/D,a,bD).
	\]
\end{prop}

\begin{proof}
	First, note that since $\alpha/\beta\in \Gamma_0(N)\cdot i\infty$, 
	$N\mid \beta$. Given the formula for the Darmon-Dasgupta measure, 
	it is sufficient to show that for all $D\mid N$, 
	\begin{align*}
		\sum_{\ell=0}^{\beta-1} \bbB_1\left(\frac{D(a+\ell)}{\beta}\right)
		\bbB_1\left(b+\frac{\alpha}{\beta}(a+\ell)\right)
		=\sum_{\ell=0}^{\beta/D -1}\bbB_1\left(\frac{D(a+\ell)}{\beta}\right)
		\bbB_1\left(Db+\frac{D\alpha}{\beta}(a+\ell)\right).
	\end{align*}
	Reindexing the left hand sum by $\ell+\tfrac{\beta}{D}\eta$ where 
	$0\le \ell<\tfrac{\beta}{D}$ and $0\le \eta< D$, we have
	\begin{align*}
		\sum_{\ell=0}^{\beta-1} \bbB_1\left(\frac{D(a+\ell)}{\beta}\right)
		\bbB_1\left(b+\frac{\alpha}{\beta}(a+\ell)\right)
		&=\sum_{\ell=0}^{\beta/D -1}\bbB_1\left(\frac{D(a+\ell)}{\beta}\right)
		\sum_{\eta=0}^{D-1}\bbB_1\left(b+\frac{\alpha}{\beta}
		(a+\ell+\eta\beta/D)\right)\\
		&=\sum_{\ell=0}^{\beta/D -1}\bbB_1\left(\frac{D(a+\ell)}{\beta}\right)
		\sum_{\eta=0}^{D-1}\bbB_1\left(b+\frac{\alpha}{\beta}
		(a+\ell)+\frac{\eta\alpha}{D}\right).
	\end{align*}
	Since $D\mid \beta$ and $(\alpha,D)=1$, the sum over
	$\eta$ is simply the distribution relation for $\bbB_1$:
	\[
		\sum_{\eta=0}^{D-1}\bbB_1\left(b+\frac{\alpha}{\beta}
		(a+\ell)+\frac{\eta}{D}\right)
		=\bbB_1\left(Db+\frac{D\alpha}{\beta}(a+\ell)\right).
	\]
\end{proof}

The following will be a recurring distribution for the remainder of 
the paper.
\begin{defn}
	We define $\Psi,\Psi^\delta\in \cD(V_\sN',\bbQ)$ as the distributions 
	$\Psi=\pi_c+\frac{1}{2}(\delta_0\otimes \bbB_1)_c$ and 
	$\Psi^\delta(f)=\sum_{D\mid N} n_D\cdot \Psi(f^D)$ for all $f\in \cS(V_\sN')$.
\end{defn}

We now compute $\DedeRa^\delta$. For all $\gamma\in G_\sN(N)$,
$\DedeRa^\delta(Id,\gamma)$ being a $\bbZ$-valued distribution, so
we can compute its value on test functions by taking a limit $z\to 
\gamma\cdot i\infty$ approaching the cusp vertically down from $i\infty$.
Recall that for a test function $f$, $f^D:=f\mid LR(D)$.
\begin{align*}
	\DedeRa^\delta(Id,\gamma)(f)&=\lim_{z\to\gamma i\infty}
	\tilde{E}^\delta(f\mid \gamma)(\gamma^{-1}z)-\tilde{E}^\delta(f)(z)\\
	&=\lim_{z\to \gamma i\infty}\sum_{D\mid N}n_D\left(
	(\pi_c+	\frac{1}{2\pi i} (\delta_0\otimes \Cyclog)_c)((f \mid\gamma - f)^D)\right.\\
	&\left. +\frac{1}{2\pi i} \int_{i\infty}^{\gamma^{-1}z} \dlog(u_c((f\mid \gamma)^D)
	(D\gamma^{-1}\tau))
	-\frac{1}{2\pi i} \int_{i\infty}^{z} \dlog(u_c(f^D)(D\tau))\right).
\end{align*}

Since only the line integrals are dependent on $z\in \cH$, 
the limit leaves us with an entirely self-contained expression.
\begin{align*}
	\DedeRa^\delta(Id,\gamma)(f)&=\sum_{D\mid N}n_D\left((\pi_c+
	\frac{1}{2\pi i} (\delta_0\otimes \Cyclog)_c)((f\mid \gamma - f)^D)
	+\frac{1}{2\pi i} \int_{\gamma i\infty}^{i\infty} \dlog(u_c(f^D)(D\tau))\right)
\end{align*}

If $\gamma\in T_\sN$, $\pi_c$ and $\delta_0\otimes \Cyclog$ are fixed
by $\gamma$ while $\gamma i\infty=i\infty$, so $\DedeRa^\delta(Id,\gamma)=0$.
Since $T_\sN$ and $\Gamma_0(N)$ generate $G_\sN(N)$, it is sufficient
to consider when $\gamma\in \Gamma_0(N)$.

Let $\gamma\in \Gamma_0(N)$ with $\gamma\cdot i\infty=\alpha/\beta$ and 
let $f$ be of the form $[(a,b)+\bbZ^2]$ (making $f^D=[(a,bD)+\bbZ^2]$).
Corollary~\ref{coreisenperiod} along with the observation that 
$\Cyclog - \RCyclog = \pi i B_1$ leaves us with the following proposition.

\begin{prop}\label{propDedeRadelta}
	For all $\gamma\in T_\sN$, $\DedeRa^\delta(Id,\gamma)=0$.
	For all $\gamma\in \Gamma_0(N)$ and $f=[(a,b)+\bbZ^2]\in \cS(V_\sN')$, 
	\begin{align*}
		\DedeRa^\delta(Id,\gamma)(f)&=\Psi^\delta(f\mid \gamma-f)
		+\sum_{D\mid N} n_D
		\left( c^2\cdot C(\alpha,\beta/D,a,bD)
		-C(\alpha,\beta/D,ac,bcD)\right).
	\end{align*}
	Since $G_\sN(N)$ is generated by $T_\sN$ and $\Gamma_0(N)$, 
	this data completely characterizes the cocycle $\DedeRa^\delta$.
\end{prop}

Let $\DedeRa^{p\delta}\in Z^1(G_\sP,\cD'(\Vadele{\sP},\bbZ))$ be the cocycle 
\[
	\DedeRa^{p\delta}=(i_\sP^\sN)_\ast(\DedeRa^\delta),
\]
which is a cocycle representative of $\mu_{DR}^{p\delta}$. We can visually observe
now that for all $\gamma\in G_\sP(N)$ and $f\in \cS(V_\sP')$, we have
\begin{align*}
	12\cdot\DedeRa^{p\delta}(Id,\gamma)(f)&+(\mu_{DD,\delta})_c(Id,\gamma)(f)
	=12\cdot\Psi^\delta(f\mid \gamma -f).
\end{align*}
To prove Theorem~\ref{thmmaincomparison}, it remains to show
that $([\gamma]-1)\ast 12\cdot \Psi^\delta$ is a coboundary,
i.e. $12\cdot \Psi^\delta$ is integer valued.

\begin{prop}\label{propcoboundaryerror}
	$12\cdot\Psi^\delta$ is an integer valued distribution.
\end{prop}

\begin{proof}
	Let $(\bbB_1)_c$ be the rank 1 $\bbZ_\sS^\times$-invariant distribution
	such that $(\bbB_1)_c([a+\bbZ])=\bbB_1(ac)$. By a simple
	calculation, we have the equality
	\[
		\Psi=-\frac{c}{2}(\delta_0\otimes(c\bbB_1-(\bbB_1)_c))
		-\frac{1}{2}((c\bbB_1-(\bbB_1)_c)\otimes \delta_0)
		+(c\bbB_1-(\bbB_1)_c)\otimes (\bbB_1)_c.
	\]
	Since $\Psi$ is invariant under scalar matrices, we can reduce 
	our analysis to when $f=[(a,b)+\bbZ^2]\in \cS_\bbZ(V_\sS')$. If 
	$a\in \bbZ$ or $b\in\bbZ$, the equality above tells us that
	$12\cdot \Psi(f)\in\bbZ$. 

	If $a,b\notin\bbZ$, we have for each $D\mid N$,
	\begin{align*}
		12\cdot\Psi(f^D)&=12(cB_1(\FPart{a})-B_1(\FPart{ac}))
		\cdot B_1(\FPart{bDc}).
	\end{align*}
	We notice that $A:=cB_1(\FPart{a})-B_1(\FPart{ac})$ is an integer
	for all $a\in\bbQ$, so
	\begin{align*}
		12\cdot\Psi^\delta(f)&=12\sum_{D\mid N} n_D \cdot A\cdot 
		B_1(\FPart{bDc})\\
		&= 12 \cdot A \sum_{D\mid N} n_D\cdot 
		\left(\FPart{bDc}-\frac{1}{2}\right)\\
		&\equiv A\cdot\sum_{D\mid N} n_D 
		\cdot (12bDc-6)\pmod{\bbZ}\\
		&=A\left(12bc\sum_{D\mid N} n_D\cdot D
		-6\sum_{D\mid N}n_D\right)\\
		&=-6A\sum_{D\mid N}n_D\in\bbZ.
	\end{align*}
\end{proof}

\begin{rem}\label{remdivisor2}
	In \cite{DD}, the $p$-adic distributions are motivated by taking
	periods of the dlog of 
	\[
		\Delta^\delta(\tau)=\prod_{D\mid N}\Delta(D\tau)^{n_D}.
	\]
	The conditions $\sum_{D\mid N} n_D= \sum_{D\mid N} n_D\cdot D=0$
	are posed to ensure $\Delta^\delta$ is a modular unit
	with no pole or zero at $i\infty$. Removing the condition 
	$\sum_{D\mid N} n_D=0$ obscures	the modular unit $\Delta^\delta$,
	but we one can still recover a $p$-smoothed modular
	unit from $\mu_{KS}^\delta$ as follows.

	Assume that $\delta$ satisfies the condition $\sum n_D\cdot D=0$ and
	for all $i,j\in \{0,\ldots,p-1\}$, let $f_{i,j}:=[(i/p,j/p)+\bbZ^2]$
	and $f:=\sum_{i,j} f_{i,j}$ where the sum is taken over all $i,j$ 
	excluding $i=0$. For all $D\mid N$, we have $f^D=f$, so 
	\begin{align*}
		\mu_{KS}^\delta(f)(\tau)
		= \prod_{D\mid N} \mu_{KS}(f)(D\tau)^{n_D}
		=\prod_{D\mid N} \zeta^{\pi_c(f)}\cdot u_c(f)(D\tau).
	\end{align*}
	However, since $f$ is supported away from $\bbZ\times\bbQ$, the product
	over $D$ of the $\zeta$-exponent is precisely $\zeta^{12\Psi^\delta(f)}$ 
	which is $1$ as seen in Proposition~\ref{propcoboundaryerror}. In addition, 
	$u_c(f)(\tau)=u(f)(\tau)^{c^2-1}$ since $c$ is coprime to $p$. Now, we
	are left with 
	\[
		\mu_{KS}^\delta(f)(\tau)=\prod_{D\mid N}u(f)(D\tau)^{c^2-1}.
	\]
	Taking the product of $u(f_{i,j})(\tau)$ over $i,j$ yields
	\begin{align*}
		u(f)(\tau) 
		&= \prod_{n=1}^\infty\frac{ (1-q^n)^2}{(1-q^{pn})^2}.
	\end{align*}
	Letting $\eta$ be the Dedekind $\eta$-function and defining 
	$\eta^\delta(\tau)=\prod_{D\mid N} \eta(D\tau)^{n_D}$,
	we see that
	\[
		\mu^\delta_{KS}(f)(\tau)=
		\left(\frac{\eta^\delta(\tau)}{\eta^\delta(p\tau)}\right)^{2(c^2-1)}
		=\left(\frac{\Delta^\delta(\tau)}{\Delta^\delta(p\tau).}\right)^{(c^2-1)/12}
	\]

	One way to interpret these modular units arising without the 
	degree zero condition on $\delta$ is noticing that at each prime,
	$\mu^\delta_{KS}$ contains $\Delta$ smoothed by the group 
	algebra element $\delta\cdot[LR(p)]$
	which is a degree zero divisor that still satisfies the equation
	$\sum_{D\mid N} n_D(D-pD)=(1-p)\sum_{D\mid N}n_D\cdot D=0$. This
	explains why Dasgupta's algorithm for computing elliptic units
	in \cite{Das} does not require the degree zero condition as remarked
	in \cite{DK}.
\end{rem}

\subsection{Explicit Dedekind-Rademacher Cocycle}
\label{subsectionexplicitdr}

Unlike the $\delta$-smoothed case, to compute $\DedeRa_\sS$,
we need to contend with a nonzero $q$-order of $\mu_{KS}$. 
However, the fact that we have no congruence condition on 
$G_\sS$ means we have an explicit set of
generators for $G_\sS$, namely the diagonal matrices $T_\sS$ and the 
special matrices $T$ and $S$. As in \S~\ref{subsectionexplicitdeltadr},
we have the following equality for all $\gamma\in G_\sS$, $f\in \cS(V_\sS')$,
and $z\in \cH$.
\[
	\DedeRa_\sS(Id,\gamma)(f)(z)=
	\lim_{z\to \gamma i\infty} \tilde{E}(f\mid \gamma)(\gamma^{-1} z)
	-\tilde{E}(f)(z)
\]

Consider $\DedeRa_\sS(Id,UL(m))(f)$. Since $UL(m)\cdot i\infty = i\infty$, 
the line integral portions of $\DedeRa_\sS(Id,UL(m))(f)$ vanishes when we 
take the limit.
\begin{align*}
	\DedeRa_\sS(Id,UL(m))(f)=\lim_{z\to i\infty}\rho(f\mid UL(m))(m^{-1}z)
	-\rho(f)(z)+\frac{1}{2\pi i}(\delta_0\otimes \Cyclog)_c(f\mid UL(m) - f).
\end{align*}
Since $\delta_0$ is scaling-invariant, the third term vanishes.
Expanding the $\rho$-terms yields
\begin{align*}
	\lim_{z\to i\infty} &\frac{z}{2}\left(m^{-1}(\bbB_2\otimes \bbB_0)_c
	(f\mid UL(m))-(\bbB_2\otimes \bbB_0)_c(f)\right)+ \pi_c(f\mid UL(m)-f).
\end{align*}
Since the distribution $B_1$ is scaling-invariant, the last term also vanishes. 
Lastly, $m^{-1}(\bbB_2\otimes \bbB_0)=(m^{-1}\bbB_2)\otimes \bbB_0$, so
using the distribution relation on $\bbB_2$ makes the remaining two terms
cancel out. Thus, $\DedeRa_\sS(Id,UL(m))=0$.

For $\DedeRa_\sS(Id,LR(m))$, $LR(m)\cdot i\infty=i\infty$, so the scaling 
invariance of $\Cyclog$ and the distribution relation of $\bbB_0$ similarly 
yield $\DedeRa_\sS(Id,LR(m))=0$.

Recall that $\tilde{E}$ is scalar-invariant, so without loss of
generality, we take $f=[(a,b)+\bbZ^2]\in \cS_{\bbZ^2}(V_\sN')$ for 
the rest of this section.
We now compute $\DedeRa_\sS(Id,T)$. $T$ stabilizes $i\infty$, so the line 
integral part of $\DedeRa_\sS(Id,T)$ similarly vanishes as $z\to i\infty$.
\begin{align*}
	\DedeRa_\sS(Id,T)(f)=\lim_{z\to i\infty}\rho(f\mid T)(z-1)-\rho(f)(z)+\frac{1}{2\pi i}
	(\delta_0\otimes \Cyclog)_c(f\mid T - f).
\end{align*}
Since $\delta_0\otimes \Cyclog$ is fixed by $T$, 
$(\delta_0\otimes \Cyclog)_c(f\mid T-f)=0$. Expanding 
the remaining terms, we have
\begin{align*}
	\lim_{z\to i\infty} &\frac{z-1}{2}
	(\bbB_2\otimes \bbB_0)_c(f\mid T)-\frac{z}{2}
	(\bbB_2\otimes \bbB_0)_c(f) + \pi_c(f\mid T-f).
\end{align*}
Since $\bbB_0$ is translation invariant, $(\bbB_2\otimes \bbB_0)_c(f\mid T-f)=0$,
which eliminates all occurences of $z$ in the limit. For the sake of uniformity,
we make the observation that $\pi_c(f\mid T-f)=\Psi(f\mid T-f)$, leaving us with:
\begin{equation}
	\DedeRa_\sS(Id,T)(f)=-\frac{1}{2}(\bbB_2\otimes \bbB_0)_c(f\mid T)
	+ \Psi(f\mid T-f)
\end{equation}

Finally, to compute the measure $\DedeRa_\sS(Id,S)$, we will require the
following proposition regarding Mellin transforms.
\begin{prop}[\cite{Stevens},Proposition 2.1.2(b)]\label{proppartialperiod}
	Let $F$ be a weight 2 modular form of level $N$ with 
	constant term $a_0(F)$. Let $\tilde{F}$ be 
	the holomorphic function $F-a_0$. Then, 
	\begin{align*}
		\sM(F,s)&=\int_i^{i\infty} \tilde{F}(\tau)\cdot y^{s-1}\,\dtau
		-\int_i^{i\infty} \tilde{F\mid S}(\tau)\cdot y^{1-s}\,\dtau\\
		&+i\left( \frac{a_0(F\mid S)}{2-s}-\frac{a_0(F)}{s} \right).
	\end{align*}
\end{prop}

We elect to compute $\DedeRa_\sS(Id,S)(f)$ by evaluating the holomorphic 
function $\tilde{E}(f\mid S)(S^{-1}z)-\tilde{E}(f)(z)$ at $z=i$. Expanding out 
definitions, we have
\begin{align*}
	\DedeRa_\sS(Id,S)(f)
	&=\tilde{E}(f\mid S)(i)-\tilde{E}(f)(i)\\
	&=\frac{i}{2}(\bbB_2\otimes \bbB_0)_c(f\mid S-f)
	+\pi_c(f\mid S-f)+ \frac{1}{2\pi i}(\delta_0\otimes \Cyclog)_c(f\mid S-f)\\
	&-\frac{1}{2\pi i} \int_{i}^{i\infty} \dlog u_c(f\mid S-f).
\end{align*}
By Proposition~2.4.2 and Proposition~2.5.1bii of \cite{Stevens},
Proposition~\ref{proppartialperiod} specializes to the following 
equality when $s=1$.
\[
	\frac{1}{2\pi i}\int_0^{i\infty}\dlog(u_c(f))
	=-\frac{1}{2\pi i}\int_i^{i\infty}\dlog(u_c(f\mid S-f))
	+\frac{i}{2}(\bbB_2\otimes \bbB_0)_c(f\mid S-f).
\]
[The distribution denoted $\phi_{(a,b)}$ in \cite{Stevens} is exactly 
$(\bbB_2\otimes \bbB_0)(f)+\tfrac{1}{2\pi i}(\tfrac{d}{d\tau}u(f))/u(f)$.]
Applying Proposition~\ref{thmeisenperiod} to the left of the above
equality and applying the result to $\DedeRa_\sS(Id,S)(f)$ leaves us 
with the following.
\[
	\DedeRa_\sS(Id,S)(f)
	=(\bbB_1\otimes\bbB_1)_c(f)+\Psi(f\mid S-f)
\]

To summarize, we have proven the following theorem:

\begin{thm}\label{thmdrformula}
	The Dedekind Rademacher cocycle is the unique cocycle $\DedeRa_\sS\in 
	Z^1(G_\sS,\cD'(\Vadele{\sS},\bbZ))$ characterized by the following 
	properties:
	\begin{enumerate}
		\item For all $\gamma\in T_\sS$, $\DedeRa_\sS(Id,\gamma)=0$,
		\item $	\DedeRa_\sS(Id,T)(f)
				=-\frac{1}{2}(\bbB_2\otimes \bbB_0)_c(f\mid T)
				+\Psi(f\mid T-f)$.
		\item $	\DedeRa_\sS(Id,S)(f)=
				(\bbB_1\otimes\bbB_1)_c(f)+\Psi(f\mid S-f).$
	\end{enumerate}
\end{thm}

\begin{rem}\label{rempsicoboundary}
	We notice from the above that the $\Psi$-terms seems to be a 
	coboundary. However, $\Psi$ is not an integer-valued 
	distribution. This sheds light on the $\DedeRa^\delta$ case
	where the proof of Theorem~\ref{thmmaincomparison} boiled 
	down to showing that the cocycles differed by a 
	$\Psi^\delta$-term. Proposition~\ref{propcoboundaryerror}
	then amounted to showing that smoothing $\Psi$ by $\delta$
	yields an integral distribution $\Psi^\delta$.
\end{rem}

\bibliography{ksdist}
\end{document}